\documentclass[12pt,reqno]{amsart}

\usepackage{amssymb} 
\usepackage{mathtools}
\usepackage{mathrsfs} 
\usepackage{bm} 
\usepackage[hidelinks]{hyperref} 
\usepackage{esint} 
\usepackage{appendix}
\usepackage{cite} 
\usepackage{cleveref} 

\usepackage [english]{babel}
\usepackage [autostyle, english = american]{csquotes}
\MakeOuterQuote{"}

\calclayout

\numberwithin{equation}{section}
\newtheorem{theorem}{Theorem}[section]
\newtheorem{corollary}[theorem]{Corollary}
\newtheorem{lemma}[theorem]{Lemma}
\newtheorem{proposition}[theorem]{Proposition}
\theoremstyle{definition}
\newtheorem{definition}{Definition}[section]
\newtheorem{remark}{Remark}[section]
\newcommand\R{\mathbb R}
\newcommand\divg{\mathop{\mbox{\rm div}}}
\newcommand{\lagr}[1]{\boldsymbol{\mathbf{#1}}}
\newcommand\adj{\mathop{\mbox{\rm adj}}}
\newcommand\supp{\mathop{\mbox{\rm supp}}}

\title[Maximal regularity for inhomogeneous Navier-Stokes equations]{Maximal $L^1$ regularity for solutions to inhomogeneous incompressible Navier-Stokes equations}

\author{Huan XU}


\AtEndDocument{%
  \par
  \medskip
  \begin{tabular}{@{}l@{}}%
    \textsc{Department of Mathematics, Auburn University, Auburn, AL 36849, USA}\\
    \textit{E-mail address}: \texttt{hzx0016@auburn.edu}
  \end{tabular}}

\begin{document}

\maketitle

\begin{abstract}
This paper is devoted to the maximal $L^1$ regularity and asymptotic behavior for solutions to the inhomogeneous incompressible Navier-Stokes equations under a scaling-invariant smallness assumption on the initial velocity. We obtain a new global $L^1$-in-time estimate for the Lipschitz seminorm of the velocity field without any smallness assumption on the fluctuation of the initial density. In the derivation of this estimate, we study the maximal $L^1$ regularity for a linear Stokes system with variable coefficients. The analysis tools are a use of the semigroup generated by a generalized Stokes operator to characterize some Besov norms and a new gradient estimate for a class of second-order elliptic equations of divergence form. Our method might be used to study some other issues arising from incompressible or compressible viscous fluids.
\end{abstract}

\bigskip


\bigskip

\section{Introduction}

\bigskip

In many problems arsing from partial differential equations in fluid mechanics, such as global existence, regularity, uniqueness, stability, etc., the heart of the matter is to bound the $L^1$-in-time integral of the Lipschitz seminorm of the velocity field.  In this work, we investigate this core issue for the three-dimensional inhomogeneous incompressible Navier-Stokes equations
\begin{eqnarray}\label{INS}
\left\{\begin{aligned}
&\partial_t\rho+u\cdot\nabla\rho =0, &\mathrm{in}\ (0,\infty)\times\R^3,\\
&\rho(\partial_tu+u\cdot\nabla u)-\Delta u+\nabla P=0,\ &\mathrm{in}\ (0,\infty)\times\R^3,\\
&\divg u=0, &\mathrm{in}\ (0,\infty)\times\R^3,\\
&(\rho,u)|_{t=0}=(\rho_0,u_0), &\mathrm{on}\ \R^3.
\end{aligned}\right.
\end{eqnarray}
In the above system, $\rho$, $u$ and $P$ stand for the density, velocity field and pressure of the fluids, respectively. The viscosity coefficient is a constant normalized to $1$. This system can model the motion of incompressible flows of mixing fluids with different densities, or of fluid flows with melted substances inside it. The inhomogeneity attributes to the presence of the variable density, and \eqref{INS} becomes the classical Navier-Stokes equations if $\rho\equiv1$. Throughout this paper, the initial density is assumed to be bounded and bounded from below, namely,
\begin{align}\label{initial density bounds}
m\le\rho_0(x)\le\frac1m,\ \ \mathrm{a.e.}\ x\in\R^3
\end{align}
for some constant $m\in(0,1]$. In view of the continuity equation of \eqref{INS}, the density $\rho(t,x)$ satisfies these bounds for all times as long as the velocity field is regular, namely,
\begin{align}\label{density bounds}
m\le\rho(t,x)\le\frac1m,\ \ \forall t\ge0,\ \mathrm{a.e.}\ x\in\R^3.
\end{align}

To be precise, we are concerned with the global-in-time estimate
\begin{align}\label{global L1Linfty estimate for gradient of velocity}
\int_0^\infty\|\nabla u(t)\|_{\infty}\,dt\ll1
\end{align}
when the initial velocity $u_0$ satisfies certain smallness condition. To our knowledge, if the fluctuation of the density is not small, there is not much evidence in the literature that supports the validity of \eqref{global L1Linfty estimate for gradient of velocity}. For example, it was proved in \cite{chen jde 2016,paicu cpde 2013} that the quantity $\int_0^T\|\nabla u(t)\|_{\infty}\,dt$ has a polynomial growth in time provided that the initial velocity $u_0$ belongs to some Sobolev space $H^s(\R^3)$ with $s>\frac12$, and that $u_0$ satisfies a scaling-invariant smallness condition. However, such growth could possibly be removed from the point of view of equivalent characterizations of norms. To see this, we suppose that $u$ is the solution to the classical heat equation $\partial_tu-\Delta u=0$ with initial value $u_0$. Then the estimate \eqref{global L1Linfty estimate for gradient of velocity} dictates the smallest norm to be used to measure the smallness of the initial data because of the equivalence of Besov norms
\begin{align*}
\int_0^\infty\|\nabla u(t)\|_{\infty}\,dt\simeq\|u_0\|_{\Dot{B}_{\infty,1}^{-1}}.
\end{align*}
The above equivalent characterization is classical and can be proved, for example, by applying {\cite[Lemma~2.4]{BCD}} and Lemma \ref{characterization via classic heat kernel}. In general, the space $\Dot{B}_{\infty,1}^{-1}(\R^3)$ is too rough in order for a nonlinear equation to be well-posed in it. Then we have to replace it with smaller spaces $\Dot{B}_{p,1}^{3/p-1}(\R^3)$, $1\le p<\infty$, so that there is still hope for \eqref{global L1Linfty estimate for gradient of velocity} to be true. In fact, the main contribution of this work is the estimate
\begin{align*}
\int_0^\infty\|\nabla u(t)\|_{\infty}\,dt\lesssim\|u_0\|_{\Dot{B}_{2,1}^{1/2}}
\end{align*}
if the right side is small.

It is worth noting that the scaling invariance of \eqref{INS} also suggests the use of the $\Dot{B}_{p,1}^{3/p-1}$ norms for the velocity. For any $\lambda>0$, it is easy to see that \eqref{INS} is invariant under the scaling
\begin{align*}
(\rho,u,P)(t,x)\rightsquigarrow(\rho_\lambda,u_\lambda,P_\lambda)(t,x)\vcentcolon=(\rho,\lambda u,\lambda^2 P)(\lambda^2t,\lambda x).
\end{align*}
So the smallness condition on $u_0$ makes sense if it is measured by a norm which is invariant under the scaling $u_0(\cdot)\rightsquigarrow\lambda u_0(\lambda\cdot)$. And we do have $\|\lambda u_0(\lambda\cdot)\|_{\Dot{B}_{p,1}^{3/p-1}}\simeq\|u_0\|_{\Dot{B}_{p,1}^{3/p-1}}$. Nowadays, the spaces $\Dot{B}_{q,1}^{3/q}(\R^3)\times\Dot{B}_{p,1}^{3/p-1}(\R^3)$ (for $(\rho_0,u_0)$) are called {\it critical spaces} for \eqref{INS}. 

The desired estimate \eqref{global L1Linfty estimate for gradient of velocity} plays a specific important role in solving hyperbolic-parabolic systems arsing from fluid dynamics (specifically, when the system has a convection-diffusion structure). It enables us to rewrite \eqref{INS} in its Lagrangian formulation which is a parabolic system only (see \cite{danchin cpam 2012,danchin arma 2013}). Now the main task is to establish a maximal $L^1$-in-time regularity for the linear part of the Lagrangian formulation in critical spaces. However, this is a challenging problem because the velocity and the pressure are strongly coupled in the presence of the variable density.

Before we present the main results of the present paper, let us review some known results for the existence and uniqueness of solutions to \eqref{INS}. Global weak solutions with finite energy were first obtained by Kazhikhov \cite{kazihov 1974} under the assumption that the initial density $\rho_0$ has a positive infimum. Several improvements can be found in \cite{simon siam 1990,Lions book,desjardins DIE 1997}. The main estimate for weak solutions is the energy inequality
\begin{equation*}
\|\sqrt{\rho}u(t)\|_{L^2}^{2} +2\int_0^t \|\nabla u(\tau)\|_{L^2}^{2}d\tau \le\|\sqrt{\rho_0}u_0\|_{L^2}^{2}.    
\end{equation*}
This estimate is far from enough to prove the uniqueness of weak solutions in $3$-D. Ladyzhenskaya and Solonnikov \cite{ladyzhenskaya} initiated the studies for unique solvability of \eqref{INS}
in a bounded domain with homogeneous Dirichlet boundary condition for $u$. In the last two decades, a large amount of work was devoted to the well-posedness of \eqref{INS} under minimum regularity assumptions on the data. Firstly, Danchin \cite{danchin edinburgh 2003} constructed a unique strong solution to \eqref{INS} in the critical space $(L^\infty(\R^3)\cap\Dot{B}_{2,\infty}^{3/2}(\R^3))\times\Dot{B}_{2,1}^{1/2}(\R^3)$ in the case when the initial density is close to a constant. Later, many authors tried to improve Danchin's result to allow different Lebesgue indices of the critical spaces, or to remove the smallness assumption on the initial density (see \cite{abidi rmi 2007,abidi aif 2007,abidi arma 2012, abidi jmpa 2013,danchin cpam 2012,burtea apde 2017,zhai jde 2017}). Secondly, it is interesting to lower the regularity of the density to allow discontinuity. A well-posedness result with only bounded density would demonstrate that the motion of a mixture of two incompressible fluids with different densities can be modeled by \eqref{INS}. One can see \cite{danchin arma 2013,huang arma 2013,paicu cpde 2013,chen jde 2016} for the work towards this direction. Now one might ask whether a small initial velocity in some critical space and a discontinuous bounded initial density can generate a unique global-in-time solution to \eqref{INS}. To our knowledge, this question is not settled. Nevertheless, in his recent paper \cite{zhang ping adv 2020}, Zhang established a global existence result in which $\rho_0$ merely satisfies \eqref{initial density bounds} and $\|u_0\|_{\Dot{B}_{2,1}^{1/2}}$ is small, yet the uniqueness is not known due to the lack of \eqref{global L1Linfty estimate for gradient of velocity} (even the local-in-time estimate is not know). Finally, we remark that our results will be based on the assumption \eqref{initial density bounds}. In case one is interested in the case when vacuum state is allowed, we refer to a recent paper \cite{danchin cpam 2019} and the references therein.

\subsection{Main results}

Our first main theorem concerns the global existence and maximal $L^1$ regularity of solutions to \eqref{INS} provided the initial velocity is small in the Besov space $\dot{B}_{2,1}^{1/2}(\R^3)$. But we do not require the initial density to stay close to an equilibrium. First, let us be clear about what it means by a solution to \eqref{INS}. Let $\mathcal{P}\dot{B}_{2,1}^{1/2}(\R^3)$ be the space that consists of all divergence free vector fields whose components belong to $\dot{B}_{2,1}^{1/2}(\R^3)$.

\begin{definition}\label{definition of strong solutions to Eulerian formulation}
Let $T\in(0,\infty]$. Suppose that $\rho_0-1\in \dot{B}_{2,1}^{3/2}(\R^3)$ and $u_0\in\mathcal{P}\dot{B}_{2,1}^{1/2}(\R^3)$. We say that $(\rho,u,\nabla P)$ is a {\it strong solution} to \eqref{INS} if
\begin{eqnarray*}
\left\{\begin{aligned}
&\rho-1\in C([0,T);\dot{B}_{2,1}^{3/2}(\R^3)), \partial_t\rho\in L_{loc}^{2}([0,T);\dot{B}_{2,1}^{1/2}(\R^3)),\\
&u\in C([0,T);\mathcal{P}\dot{B}_{2,1}^{1/2}(\R^3)),\\
&(\partial_tu,\Delta u,\nabla P)\in \left(L_{loc}^{1}([0,T);\dot{B}_{2,1}^{1/2}(\R^3))\right)^3,
\end{aligned}\right.
\end{eqnarray*}
and $(\rho,u,\nabla P)$ satisfies \eqref{INS} for a.e. $t\in(0,T)$.
\end{definition}

\begin{theorem}\label{global wellposedness}
Assume that the initial density $\rho_0$ satisfies \eqref{initial density bounds}, $\rho_0-1\in\dot{B}_{2,1}^{3/2}(\R^3)$, and the initial velocity $u_0\in\mathcal{P}\dot{B}_{2,1}^{1/2}(\R^3)$. Then there exists some $T>0$ such that \eqref{INS} has a unique local-in-time solution $(\rho,u,\nabla P)$ with $\rho$ verifying \eqref{density bounds}.

Moreover, there exists a positive constant $\varepsilon_0$ depending on $m$ and $\|\rho_0-1\|_{\dot{B}_{2,1}^{3/2}}$ such that if $u_0$ satisfies
\begin{align}\label{smallness on initial velocity}
\|u_0\|_{\dot{B}_{2,1}^{1/2}}\le\varepsilon_0,
\end{align}
then the above solution exists globally in time and verifies
\begin{align}\label{maximal regularity INS}
\|u\|_{L^\infty(\R_+;\dot{B}_{2,1}^{1/2})}+\|\Delta u,\partial_tu,\nabla P\|_{L^1(\R_+;\dot{B}_{2,1}^{1/2})}\le C_0\|u_0\|_{\dot{B}_{2,1}^{1/2}},
\end{align}
and
\begin{align}\label{global in time estimate for the density}
\|\rho-1\|_{L^\infty(\R_+;\dot{B}_{2,1}^{3/2})}\le C_1\|\rho_0-1\|_{\dot{B}_{2,1}^{3/2}},
\end{align}
where $C_0$ is a constant depending on $m$ and $\|\rho_0-1\|_{\dot{B}_{2,1}^{3/2}}$, and $C_1$ is an absolute constant.
\end{theorem}

\begin{remark}
The estimate of the $L^1(\R_+;\dot{B}_{2,1}^{1/2})$ norm in \eqref{maximal regularity INS} is a maximal regularity type estimate. To our knowledge, this is the first such result concerning maximal $L^1$ regularity for density-dependent viscous fluids without any smallness assumption on the fluctuation of the initial density.
\end{remark}

\begin{remark}
The regularity assumption on the initial density can be weakened to allow a slight discontinuity. But to simplify the exposition and avoid unpleasant
technicalities, we do not pursue the optimal regularity assumption on $\rho_0$. Nevertheless, we do not know if our method could be improved to give maximal $L^1$ regularity estimates for weak solutions to \eqref{INS} with merely measurable (or, piecewise constant) initial densities.
\end{remark}

\begin{remark}
The estimate $\|u\|_{L^\infty(\R_+;\dot{B}_{2,1}^{1/2})}\le C\|u_0\|_{\dot{B}_{2,1}^{1/2}}$ for weak solutions to \eqref{INS} has been recently obtained by Zhang in \cite{zhang ping adv 2020}, in which he also obtained some global $L^2$-in-time estimates. The initial density in \cite{zhang ping adv 2020} is merely bounded and bounded from zero. But the uniqueness of weak solutions in critical spaces is not known, unless the initial velocity field has more regularity so that one can prove a local $L^1$-in-time estimate for the Lipschitz seminorm of the velocity field (see also \cite{chen jde 2016,paicu cpde 2013}). However, the energy methods used in \cite{chen jde 2016,paicu cpde 2013,zhang ping adv 2020} are unlikely to give \eqref{maximal regularity INS}, even if the data is smooth.
\end{remark}

As an application of \eqref{maximal regularity INS}, we prove a second result concerning the long time asymptotics for the globally-defined velocity constructed in Theorem \ref{global wellposedness}.

\begin{theorem}\label{long time asymptotics}
Assume that $\rho_0$ and $u_0$ satisfy the assumptions in Theorem \ref{global wellposedness} and \eqref{smallness on initial velocity}. Let $(\rho,u,\nabla P)$ be the global-in-time solution to \eqref{INS}. Then it holds that
\begin{align}\label{large solution asymptotic to 0}
\lim_{t\rightarrow\infty}\|u(t)\|_{\dot{B}_{2,1}^{1/2}}=0.
\end{align}
\end{theorem}

\subsection{Strategy of the proof}

The proof of Theorem \ref{long time asymptotics} relies on Theorem \ref{global wellposedness}, so let us only elaborate the strategy for the proof of Theorem \ref{global wellposedness}. 

\subsubsection{Lagrangian coordinates}
The Lagrangian coordinate is a natural coordinate system at fluid motion, in which the observer follows an individual fluid parcel as it moves through space and time. It can be used to convert a free boundary problem into an equivalent problem in a fixed domain (see, e.g., \cite{solonnikov}); or to convert a coupled hyperbolic-parabolic system into a merely parabolic system (see, e.g., \cite{danchin cpam 2012,danchin arma 2013,danchin memoir 2015}).

Assume temporarily that $u=u(t,x)$ is a $C^1$ vector field (not necessarily divergence free). By virtue of Cauchy-Lipschitz theorem (see Appendix \ref{appendix1}), the unique trajectory $X(t,\cdot)$ of $u$ is a $C^1$-diffeomorphism over $\R^3$ for every $t\ge0$. Let us introduce $A(t,y)=\big(D_y X(t,y)\big)^{-1}$, $J(t,y)=\det DX(t,y)$, and $\mathscr{A}(t,y)=\adj DX(t,y)$ (the adjugate of $D X$, i.e.,  $\mathscr{A}=JA$). For any scalar function $\phi=\phi(x)$ and any vector field $v=v(x)$, it is easy to see that
\begin{align}\label{change of variable for gradient}
(\nabla\phi)\circ X=A^{T}\nabla(\phi\circ X),
\end{align}
and
\begin{align}\label{change of variable for divergence1}
(\divg v)\circ X=\mathrm{Tr}[AD(v\circ X)],
\end{align}
where $A^T$ denotes the transpose of $A$, and $\mathrm{Tr}A$ denotes the trace of $A$. On the other hand, using an integration by part argument as in the appendix of \cite{danchin cpam 2012}, we also have
\begin{align}\label{change of variable for divergence2}
(\divg v)\circ X=J^{-1}\divg(\mathscr{A}(v\circ X)).
\end{align}
This along with \eqref{change of variable for divergence1} gives the following identity
\begin{align}\label{identity for divergence}
\mathrm{Tr}[\mathscr{A}D(v\circ X)]=\divg(\mathscr{A}(v\circ X)).
\end{align}
By writing $\Delta=\divg\nabla$, we get from \eqref{change of variable for gradient} and \eqref{change of variable for divergence2} that
\begin{align}\label{change of variable for Laplacian}
(\Delta v)\circ X=J^{-1}\divg(\mathscr{A}A^{T}\nabla(v\circ X)).
\end{align}
Note that all the above equations hold even if $u$ is not divergence free. But if $\divg u=0$, we have $J\equiv1$, and thus, $A=\mathscr{A}$.

Now we introduce new unknowns
\begin{align*}
(\lagr{\rho},\lagr{u},\lagr{P})(t,y)=(\rho,u,P)\big(t,X(t,y)\big).
\end{align*}
In view of the continuity equation in \eqref{INS}, we have $\lagr{\rho}\equiv\rho_0$. Using \cref{change of variable for gradient,change of variable for divergence1,change of variable for divergence2,identity for divergence,change of variable for Laplacian} and the chain rule, one can formally convert the system \eqref{INS} into its Lagrangian formulation that reads
\begin{eqnarray}\label{Lagrangian formulation}
\left\{\begin{aligned}
&\rho_0\partial_t \lagr{u}-\divg(\mathscr{A}_{\lagr{u}}\mathscr{A}_{\lagr{u}}^{T}\nabla\lagr{u})+\mathscr{A}_{\lagr{u}}^{T}\nabla\lagr{P}=0,\\
&\divg(\mathscr{A}_{\lagr{u}} \lagr{u})=\mathrm{Tr}(\mathscr{A}_{\lagr{u}}D\lagr{u})=0,\\
&\lagr{u}|_{t=0}=u_0.
\end{aligned}\right.
\end{eqnarray}
In this new system, we associate $\mathscr{A}_{\lagr{u}}$ with the new velocity $\lagr{u}$ so that the system is closed (i.e., determined). Precisely, we denote
\begin{align}\label{definition of Au and Xu}
\mathscr{A}_{\lagr{u}}=\adj DX_{\lagr{u}},\ \ \mathrm{with}\ \ X_{\lagr{u}}(t,y)=y+\int_0^t \lagr{u}(\tau,y)\,d\tau.
\end{align}

\begin{remark}
One can write $A_{\lagr{u}}\vcentcolon=(DX_{\lagr{u}}(t,y))^{-1}$ in place of $\mathscr{A}_{\lagr{u}}$ in the "momentum" equation of \eqref{Lagrangian formulation}. But as in the work of Solonnikov \cite{solonnikov}, one should use $\mathscr{A}_{\lagr{u}}$ in the second equation. The reason is that, when linearizing \eqref{Lagrangian formulation} to seek existence, we will need the fact that \eqref{identity for divergence} is an identity, whether $u$ is divergence free or not. Of course, once the existence issue is settled, we can use either $\mathscr{A}_{\lagr{u}}$ or $A_{\lagr{u}}$ in \eqref{Lagrangian formulation}. More precisely, the linearized system of \eqref{Lagrangian formulation} reads
\begin{eqnarray}\label{linearized Lagrangian formulation}
\left\{\begin{aligned}
&\rho_0\partial_t \lagr{u}-\Delta\lagr{u}+\nabla\lagr{P}=\divg((\mathscr{A}_{\lagr{v}}\mathscr{A}_{\lagr{v}}^{T}-I)\nabla\lagr{v})+(I-\mathscr{A}_{\lagr{v}}^{T})\nabla\lagr{Q},\\
&\divg\lagr{u}=\divg((I-\mathscr{A}_{\lagr{v}})\lagr{v})=Tr((I-\mathscr{A}_{\lagr{v}})D\lagr{v}),\\
&\lagr{u}|_{t=0}=u_0.
\end{aligned}\right.
\end{eqnarray}
To obtain {\it a priori} estimates for this system, we need to write $\divg\lagr{u}$ in two different ways. To conclude, we will solve \eqref{linearized Lagrangian formulation} without assuming $\det DX_{\lagr{v}}\equiv1$.
\end{remark}

\begin{definition}\label{definition of strong solutions to Lagrangian formulation}
Let $\rho_0-1\in \dot{B}_{2,1}^{3/2}(\R^3)$ and $u_0\in\mathcal{P}\dot{B}_{2,1}^{1/2}(\R^3)$. We say that $(\lagr{u},\nabla\lagr{P})$ is a {\it strong solution} to \eqref{Lagrangian formulation} if for some $T\in(0,\infty]$,
\begin{eqnarray*}
\left\{\begin{aligned}
&\lagr{u}\in C([0,T);\dot{B}_{2,1}^{1/2}(\R^3)),\\
&\mathscr{A}_{\lagr{u}}-I\in C([0,T);\dot{B}_{2,1}^{3/2}(\R^3)),\\
&(\partial_t\lagr{u},\Delta \lagr{u},\nabla\lagr{P})\in \left(L_{loc}^{1}([0,T);\dot{B}_{2,1}^{1/2}(\R^3))\right)^3,
\end{aligned}\right.
\end{eqnarray*}
and $(\lagr{u},\nabla\lagr{P})$ satisfies \eqref{Lagrangian formulation} for a.e. $t\in(0,T)$.
\end{definition}

The justification of equivalence between Eulerian and Lagrangian formulations can be found, for example, in \cite{danchin cpam 2012,danchin memoir 2015}. For the reader's convenience, we will give the details in Appendix \ref{appendix1}. Let us now state a well-posedness result for \eqref{Lagrangian formulation}.

\begin{theorem}\label{local and global well posedness for Lagragian formulation}
Assume that the initial density $\rho_0$ satisfies \eqref{initial density bounds}, $\rho_0-1\in\dot{B}_{2,1}^{3/2}(\R^3)$, and the initial velocity $u_0\in\mathcal{P}\dot{B}_{2,1}^{1/2}(\R^3)$. Then there exists some $T>0$ such that \eqref{Lagrangian formulation} has a unique local solution $(\lagr{u},\nabla\lagr{P})$ with $\lagr{u}$ verifying
\begin{align}
\|\nabla\lagr{u}\|_{L_T^1(\dot{B}_{2,1}^{3/2})}\le c_0
\end{align}
for some positive constant $c_0$.

Moreover, there exists another constant $\varepsilon_0$ depending on $m$ and $\|\rho_0-1\|_{\dot{B}_{2,1}^{3/2}}$ so that if $u_0$ satisfies
\begin{align*}
\|u_0\|_{\dot{B}_{2,1}^{1/2}}\le\varepsilon_0,
\end{align*}
the local solution becomes globally in time and verifies
\begin{align*}
\|\lagr{u}\|_{L^\infty(\R_+;\dot{B}_{2,1}^{1/2})}+\|\Delta \lagr{u},\partial_t\lagr{u},\nabla\lagr{P}\|_{L^1(\R_+;\dot{B}_{2,1}^{1/2})}\le C_0\|u_0\|_{\dot{B}_{2,1}^{1/2}}
\end{align*}
for some constant $C_0$ depending on $m$ and $\|\rho_0-1\|_{\dot{B}_{2,1}^{3/2}}$.
\end{theorem}

\subsubsection{Maximal regularity for Stokes system}
The heart of the present paper is the well-posedness issue for the linearized system \eqref{linearized Lagrangian formulation}. To this end, we will mainly focus on the maximal $L^1$ regularity for the following linear Stokes-like system
\begin{eqnarray}\label{Stokes system}
\left\{\begin{aligned}
&\rho(x)\partial_tu-\Delta u+\nabla P=f,
&(t,x)\in\mathbb{R}^{+}\times\mathbb{R}^{3},\\
&\divg u=\divg R=g,&(t,x)\in\mathbb{R}^{+}\times\mathbb{R}^{3},\\
&u(0,x)=u_0(x),&x\in\mathbb{R}^{3}.
\end{aligned}\right.
\end{eqnarray}
Here the coefficient $\rho$ is a time-independent function that satisfies \eqref{initial density bounds} and no other assumption is needed temporarily. The system is supplemented with the compatibility condition $\divg u_0(x)=\divg R(0,x)$. For simplicity, we will take $R(0,\cdot)=0$, since this is the case in applications. 

We are going to construct solutions to \eqref{Stokes system} using the theory of $C_0$ semigroups and abstract Cauchy problem. In doing so, let us introduce a new variable
\begin{equation*}
v\vcentcolon=u-\mathcal{Q}R=u+\nabla(-\Delta)^{-1}g,  
\end{equation*}
where $\mathcal{Q}=-\nabla(-\Delta)^{-1}\divg$ is the Hodge operator. Then the system \eqref{Stokes system} can be equivalently reformulated as
\begin{eqnarray}\label{equivalent Stokes system}
\left\{\begin{aligned}
&\rho\partial_tv-\Delta v+\nabla(P-g)=f-\rho\mathcal{Q}\partial_tR,\\
&\divg v=0,\\
&v(0,x)=u_0(x).
\end{aligned}\right.
\end{eqnarray}
We shall obtain maximal $L^1$ regularity estimates for the above system. Compared with the classical Stokes system, this is a challenging problem because the velocity and the pressure are strongly coupled in the presence of the density. Let us now explain how to achieve our goal. Denote $b=\rho^{-1}$ and let $\mathcal{L}_b=-\divg(b\nabla)$. Applying $\divg b$ to the first equation of \eqref{equivalent Stokes system} and using the second equation, we see that
\begin{align*}
\mathcal{L}_b(P-g)=-\divg[b(\Delta v+f-\rho\mathcal{Q}\partial_tR)].
\end{align*}
Next, we introduce the Hodge operator
\begin{align}\label{Hodge b operator}
\mathcal{Q}_b\vcentcolon=-\nabla \mathcal{L}_b^{-1}\divg b
\end{align}
associated with $\mathcal{L}_b$, and let
\begin{align}\label{Pb operator}
\mathcal{P}_b=I-\mathcal{Q}_b.
\end{align}
Then we can write
\begin{align}\label{expression for gradient of P minus g}
\nabla(P-g)=\mathcal{Q}_b(\Delta v+f-\rho\mathcal{Q}\partial_tR). 
\end{align}
Plugging \eqref{expression for gradient of P minus g} in \eqref{equivalent Stokes system}, we hence introduce a generalized Stokes operator
\begin{align}\label{Stokes operator}
\mathcal{S}\vcentcolon=b\mathcal{P}_b\Delta,
\end{align}
so \eqref{equivalent Stokes system} can be further reformulated as an abstract Cauchy problem
\begin{eqnarray}\label{ACP Stokes}
\left\{\begin{aligned}
&\partial_tv-\mathcal{S}v=\tilde{f}\vcentcolon=b\mathcal{P}_b(f-\rho\mathcal{Q}\partial_tR),\\
&v(0,x)=u_0(x).
\end{aligned}\right.
\end{eqnarray}
We will show that the Stokes operator $\mathcal{S}$ generates a semigroup $e^{t\mathcal{S}}$ on 
\begin{align*}
\mathcal{P}L^2\vcentcolon=\{u\in L^2(\R^3,\R^3)| \divg u=0\}.
\end{align*}
So by Duhamel's principle, we can formally write the solution $v$ to \eqref{ACP Stokes} as
\begin{align}\label{Duhamel principle for v}
v(t)=e^{t\mathcal{S}}u_0+\int_0^te^{(t-\tau)\mathcal{S}}\tilde{f}(\tau)\,d\tau.
\end{align}
To obtain maximal regularity estimates for $v$, we will characterize some Besov norms for divergence free vector fields via the semigroup $e^{t\mathcal{S}}$. In fact, we are able to prove the following:

\begin{theorem}\label{characterization of besov spaces with positive regularity via stokes}
Assume that $b$ satisfies \eqref{initial density bounds}. For any $s\in(0,2)$, any $q\in[1,\infty]$, and any
\begin{align*}
u_0\in\mathcal{P}H^2\vcentcolon=\{u\in H^2(\R^3,\R^3)| \divg u=0\},
\end{align*}
we have
\begin{align*}
\|u_0\|_{\dot{B}_{2,q}^{s}}\simeq\left\|t^{-s/2}\|t\mathcal{S}e^{t\mathcal{S}}u_0\|_2\right\|_{L^q(\R_+,\frac{dt}{t})}\simeq\left\|t^{-s/2}\|t\Delta e^{t\mathcal{S}}u_0\|_2\right\|_{L^q(\R_+,\frac{dt}{t})}.
\end{align*}
\end{theorem}
\begin{remark}
Note that we do not need any regularity for the coefficient $b$ at this point.
\end{remark}

Applying Theorem \ref{characterization of besov spaces with positive regularity via stokes} (with $q=1$) to \eqref{Duhamel principle for v} gives us an {\it a priori} maximal regularity estimate for $v$:
\begin{align}\label{a priori maximal regularity estimate for v}
\|v\|_{L_T^\infty(\dot{B}_{2,1}^{s})}+\|\partial_t v,\mathcal{S}v\|_{L_T^1(\dot{B}_{2,1}^{s})}\lesssim\|u_0\|_{\dot{B}_{2,1}^{s}}+\|\tilde{f}\|_{L_T^1(\dot{B}_{2,1}^{s})}.
\end{align}
This gives good estimates for $(u,\nabla P)$ except for $\|u\|_{L_T^\infty(\dot{B}_{2,1}^{s})}$. If we apply \eqref{a priori maximal regularity estimate for v} to bound $\|u\|_{L_T^\infty(\dot{B}_{2,1}^{s})}$ directly, we need to use $\|\mathcal{Q}R\|_{L_T^\infty(\dot{B}_{2,1}^{s})}$, which would cause trouble when proving local existence of large solutions to \eqref{Lagrangian formulation}. To overcome this difficulty, we view $\nabla P$ in the first equation of \eqref{Stokes system} as a source term, and write
\begin{align}\label{write P as a source term}
\partial_tu-b\Delta u=b(f-\nabla P).
\end{align}
Then the maximal $L^1$ regularity estimate for $u$ can be obtained by equivalent characterizations of Besov norms via the semigroup $e^{tb\Delta}$.

\subsubsection{Elliptic estimates}

It remains to bound the inhomogeneous term $\tilde{f}$ in \eqref{ACP Stokes}. For this, we need to study the continuity of the operator $b\mathcal{P}_b$ on Besov spaces. In other words, we need to study the gradient estimates for solutions to the divergence form elliptic equation
\begin{align}\label{2nd order elliptic equation of divergence form}
-\mathcal{L}_bP=\divg f.
\end{align}
But this is again a difficult problem for it is well-known that $\nabla\mathcal{L}_b^{-1}\divg$ is not of Calder\'{o}n-Zygmund type. In general, $\nabla\mathcal{L}_b^{-1}\divg$ is not bounded on $L^p$ for $p$ not close enough to $2$, even if the coefficient $b$ is smooth (see \cite{jiang jmpa 2020}). In fact, in order to prove continuity of $\nabla\mathcal{L}_b^{-1}\divg$ on homogeneous function spaces, one should treat it as a zeroth-order operator. This suggests that $b$ should be in some function spaces that have the same scaling as $L^\infty$. So we once again need $b$ to be in some "critical" spaces. 

Our strategy is to use an iteration scheme to gain elliptic regularity. In the initial iteration, we prove an inequality of the form
\begin{align*}
\|\nabla\mathcal{L}_b^{-1}\divg f\|_{\dot{B}_{p_0,r}^{s_0}}\le C\|f\|_{\dot{B}_{p,r}^{s}},
\end{align*}
in which a loss of regularity is allowed, but the scalings of both Besov spaces are the same, meaning that $s_0-\frac{3}{p_0}=s-\frac{3}{p}$ with $s_0<s$. The proof relies on equivalent characterizations of Besov norms via the heat semigroup $e^{-t\mathcal{L}_b}$ (see \cite{bui adv 2012,liu fm 2015}), and the boundedness of the Riesz transform $\nabla\mathcal{L}_b^{-1/2}$ on $L^p$, $1<p\le2$ (see \cite{auscher book 1998}). In this step, we only require $b$ to satisfy \eqref{initial density bounds}. But if $b$ has more regularity (in critical spaces), the loss of regularity can be recovered via an iteration scheme. We are able to eventually prove that $\nabla\mathcal{L}_b^{-1}\divg$ is bounded on some homogeneous Besov spaces including $\dot{B}_{2,1}^{1/2}(\R^3)$. Such result is nontrivial for $\dot{B}_{2,1}^{1/2}(\R^3)$ has the same scaling as $L^{3}(\R^3)$. Indeed, our method also applies to $L^p$ elliptic gradient estimates for $p$ slightly larger than the space dimensions.

Carrying out the details of the strategy, we can prove the following maximal $L^1$ regularity theorem for the Stokes system \eqref{Stokes system}.

\begin{theorem}\label{maximal regularity for Stokes}
Let $T\in(0,\infty]$. Assume that $\rho=\rho(x)$ satisfies \eqref{initial density bounds}, $\rho-1\in\dot{B}_{2,1}^{3/2}(\R^3)$ and $u_0\in\mathcal{P}\dot{B}_{2,1}^{1/2}(\R^3)$. Suppose that $f,g,R$ satisfy
\begin{align*}
R\in C_b([0,T);\dot{B}_{2,1}^{1/2}(\R^3)),\ \ (f,\nabla g,\partial_t R)\in\left(L^1((0,T);\dot{B}_{2,1}^{1/2}(\R^3))\right)^3,
\end{align*}
$R(0)=0$, and $\divg R=g$. Then the system \eqref{Stokes system} has a unique strong solution $(u,\nabla P)$ in the class
\begin{align*}
u\in C_b([0,T);\dot{B}_{2,1}^{1/2}(\R^3)),\ (\Delta u,\partial_tu,\nabla P)\in\left(L^1((0,T);\dot{B}_{2,1}^{1/2}(\R^3))\right)^3.
\end{align*}
Moreover, there exists a constant $C$ depending on $m$ and $\|\rho-1\|_{\dot{B}_{2,1}^{3/2}}$ such that
\begin{align}\label{maximal regularity estimates for Stokes}
\|u\|_{L_T^\infty(\dot{B}_{2,1}^{1/2})}+\|\partial_tu,\Delta u,\nabla P\|_{L_T^1(\dot{B}_{2,1}^{1/2})}\le C\|u_0\|_{\dot{B}_{2,1}^{1/2}}+C\|f,\partial_t R,\nabla g\|_{L_T^1(\dot{B}_{2,1}^{1/2})}.
\end{align}
\end{theorem}

Let us conclude the introduction with a few remarks on the methodology. Overall, the proof of the estimate like \eqref{global L1Linfty estimate for gradient of velocity} requires three main ingredients: Lagrangian coordinates, critical spaces and maximal $L^1$ regularity. For a (brief) history of the development of this method, we refer to \cite{danchin edinburgh 2003,danchin jfa 2009,danchin cpam 2012,danchin memoir 2015}. However, this paper gives the first application of this method to viscous fluids with truly variable densities. In a subsequent paper, we apply this method to establish a maximal $L^1$ regularity result for a parabolic Lam\'{e} system with rough coefficients. Finally, in view of the latest application of this method to free boundary problems \cite{danchin arxiv 2020}, we have reason to believe that it can be applied to tackle more problems arising from fluid dynamics.

\bigskip

\noindent {\bf Organization of the paper:} In Section \ref{preliminaries}, we recall some prerequisites covering the abstract Cauchy problem, elliptic operators of divergence form, and the theory of homogeneous Besov spaces. Section \ref{main section} is devoted to the proof of maximal $L^1$ regularity for the Stokes system \eqref{Stokes system}. To this end, we give equivalent characterizations of homogeneous Besov norms via the semigroup generated by the generalized Stokes operator, and prove elliptic gradient estimates. Then, in Section \ref{proof of main theorems}, we apply the results established in Section \ref{main section} to prove Theorem \ref{global wellposedness} and Theorem \ref{local and global well posedness for Lagragian formulation}. The proof of Theorem \ref{long time asymptotics} will be given in Section \ref{proof of long time behavior}. In Appendix \ref{appendix1}, for the reader's convenience, we collect some useful estimates for changes of variables and verify the equivalence between Eulerian and Lagrangian formulations in the framework of homogeneous Besov spaces.

\bigskip

\noindent {\bf Notations:} Throughout the paper, the letter $C$ stands for a generic positive constant that may change from line to line. The notation $a\lesssim b$ means $a\leq Cb$ for some $C$, and $a\simeq b$ means $a\lesssim b$ and $b\lesssim a$. For two quantities $a, b$, we denote by $a\vee b$ the bigger quantity and by $a\wedge b$ the smaller one. For $p\in[1,\infty]$, the conjugate index $p'$ is determined by $\frac1p+\frac{1}{p'}=1$. We denote by $\|\cdot\|_p$ the Lebesgue $L^p$-norm. For a Banach space $X$ and $q\in[1,\infty]$, we may write $\|\cdot\|_{L_t^q(X)}$ for the norm of the space $L^q((0,t);X)$. Finally, we denote operators on Banach spaces by "mathcal" letters (e.g., $\mathcal{P}$, $\mathcal{Q}$, $\mathcal{S}$, etc.), variables in Lagrangian coordinates by bold letters (e.g., $\lagr{\rho}$, $\lagr{u}$, $\lagr{P}$, etc.).

\bigskip

\section{Preliminaries}\label{preliminaries}

\bigskip

\subsection{Semigroups and abstract Cauchy problem}


We adopt the concept that a $C_0$ semigroup $\{\mathcal{T}(t)\}_{t\ge0}$ on a Banach space $(X,\|\cdot\|)$ is called a bounded $C_0$ semigroup if $\|\mathcal{T}(t)\|\le C<\infty$ for all $t\ge0$, while it is called a contraction semigroup if $C=1$. In this paper, we take a "real" characterization of analyticity of semigroups. Such a characterization can also be taken as the definition of real analytic semigroups.
\begin{lemma}[see {\cite[Theorem~4.6]{engel book 2000}}]\label{generation theorem for analytic semigroups}
Let $\mathcal{T}(t)=e^{t\mathcal{A}}$ be a bounded $C_0$ semigroup on a Banach space $(X,\|\cdot\|)$ with infinitesimal generator $\mathcal{A}$. Then $\mathcal{T}(t)$ is a bounded analytic semigroup if and only if there is a dense subset $S$ of $X$ such that
\begin{align*}
\|t\mathcal{A}e^{t\mathcal{A}}x\|\le C\|x\|,\ \forall x\in S,\ \forall t>0.
\end{align*}
\end{lemma}
\begin{remark}
Obviously, the above inequality holds for all $x\in X$ since $\mathcal{A}$ is closed.
\end{remark}

Let $(H,\langle\cdot,\cdot\rangle)$ be a real Hilbert space. A linear operator $\mathcal{A}:D(\mathcal{A})\subset H\rightarrow H$ is called dissipative on $H$ if
\begin{align*}
\langle  \mathcal{A}x,x\rangle\le0,\ \forall x\in D(\mathcal{A}).
\end{align*}
We have the following well-known result:
\begin{theorem}\label{generation theorem of selfadjoint operator}
Let $\mathcal{A}$ be a self-adjoint operator on $H$. Then $\mathcal{A}$ generates an analytic semigroup of contraction $\{e^{t\mathcal{A}}\}_{t\ge0}$ if and only if $\mathcal{A}$ is dissipative. Moreover, $e^{t\mathcal{A}}$ is self-adjoint on $H$ for every $t\ge0$.
\end{theorem}
For the complex version of Theorem \ref{generation theorem of selfadjoint operator}, we refer to {\cite[Example~3.7.5]{arendt book 2011}} and {\cite[Corollary~3.3.9]{arendt book 2011}}. 

In applications, we will first apply Theorem \ref{generation theorem of selfadjoint operator} to construct a semigroup on $L^2$, and then extrapolate it to some other function spaces. However, it is usually not easy to identify the generator of the new semigroup. In this situation, we wish
to identify the generator restricted on a dense subspace of its domain. Recall that a subspace $Y$ of the domain $D(\mathcal{A})$ of a linear operator $\mathcal{A}:D(\mathcal{A})\subset X\rightarrow X$ is called a core for $\mathcal{A}$ if $Y$ is dense in $D(\mathcal{A})$ for the graph norm $\|x\|_{D(\mathcal{A})}\vcentcolon=\|x\|+\|\mathcal{A}x\|$. In other words, $Y$ is a core for $\mathcal{A}$ if and only if $\mathcal{A}$ is the closure of $\mathcal{A}|_{Y}$. The next result gives a useful sufficient condition for a subspace to be a core for the generator.

\begin{lemma}[see {\cite[p. 53]{engel book 2000}}]\label{identify a core}
Let $\mathcal{A}$ be the infinitesimal generator of a $C_0$ semigroup $\mathcal{T}(t)$ on $X$. If $Y\subset D(\mathcal{A})$ is a dense subspace of $X$ and invariant under $\mathcal{T}(t)$ (i.e., $\mathcal{T}(t)Y\subset Y$), then $Y$ is a core for $\mathcal{A}$.
\end{lemma}

Next, we recall shortly how to use semigroups to solve abstract Cauchy problems. Suppose that $\mathcal{A}$ is the infinitesimal generator of a $C_0$ semigroup $e^{t\mathcal{A}}$ on a Banach space $(X,\|\cdot\|)$. We are concerned with the well-posedness of the inhomogeneous abstract Cauchy problem
\begin{eqnarray}\label{ACP}
\left\{\begin{aligned}
&u'(t)-\mathcal{A}u(t)=f(t),\ \ 0<t\le T,\\
&u(0)=u_0.
\end{aligned}\right.
\end{eqnarray}
We assume that $u_0\in X$ and the inhomogeneous term $f$ only belongs to $L^1((0,T);X)$. Then \eqref{ACP} always has a unique {\it mild solution} $u\in C([0,T];X)$ given by the formula
\begin{equation*}
u(t)=e^{t\mathcal{A}}u_0+\int_0^t e^{(t-\tau)\mathcal{A}}f(\tau)\,d\tau.
\end{equation*}
A continuous function $u$ is called a {\it strong solution} if $u\in W^{1,1}((0,T);X)\cap L^1((0,T);D(\mathcal{A}))$ satisfies \eqref{ACP} for a.e. $t\in(0,T)$. A strong solution is also a mild solution. Conversely, a mild solution with suitable regularity becomes a strong one.
\begin{lemma}[see {\cite[Theorem~2.9]{pazy book 1983}}]\label{mild to strong}
Let $u\in C([0,T];X)$ be a mild solution to \eqref{ACP}. If $u\in W^{1,1}((0,T),X)$, or $u\in L^1((0,T),D(\mathcal{A}))$, then $u$ is a strong solution.
\end{lemma}

\subsection{Elliptic operators of divergence form}

We recall some known results concerning elliptic equations of divergence form, heat kernels and Riesz transforms. The second order divergence form elliptic equation to be considered is given by \eqref{2nd order elliptic equation of divergence form} with $b$ satisfying \eqref{initial density bounds}. But we remark that all the subsequent results concerning $\mathcal{L}_b$ remain true if the coefficient $b$ is replaced by a real symmetric matrix that satisfies a uniform ellipticity condition. To simplify the notation, we write $\mathcal{L}=\mathcal{L}_b$.

Let $\dot{H}^1(\R^3)$ be the space of distributions having a square-integrable gradient, equipped with the inner product $\int\nabla u\cdot\nabla v\,dx$. We adopt the convention that two functions in $\dot{H}^1(\R^3)$ are identical if their difference is a constant. For any $f\in L^2(\R^3;\R^3)$, as a consequence of the Lax-Milgram theorem, the equation \eqref{2nd order elliptic equation of divergence form} has a unique weak solution $P\in\dot{H}^1(\R^3)$ satisfying $\|\nabla P\|_2\le\frac{1}{m}\|f\|_2$.

Let $D(\mathcal{L})=\{u\in H^1(\R^3)| \mathcal{L}u\in L^2(\R^3)\}$. Then $-\mathcal{L}:D(\mathcal{L})\subset L^2\rightarrow L^2$ is a dissipative self-adjoint operator that generates an analytic semigroup of contraction $e^{-t\mathcal{L}}$. The maximal accretive square root of $\mathcal{L}$ is given by the formula
\begin{align*}
\mathcal{L}^{1/2}u=\frac{2}{\sqrt{\pi}}\int_0^\infty e^{-t^2\mathcal{L}}\mathcal{L}u\,dt,\ \ u\in D(\mathcal{L}).
\end{align*}
The above integral converges normally in $L^2$ since $\|e^{-t^2\mathcal{L}}\mathcal{L}u\|\le \|\mathcal{L}u\|\wedge t^{-2}\|u\|$. The domain $D(\mathcal{L}^{1/2})$ of $\mathcal{L}^{1/2}$ coincides with $H^1(\R^3)$ and it holds that $\|\mathcal{L}^{1/2}u\|\simeq\|\nabla u\|$ for $u\in H^1(\R^3)$. Indeed, $\mathcal{L}^{1/2}$ extends to a continuous operator from $\dot{H}^1(\R^3)$ to $L^2(\R^3)$ with a continuous inverse (denoted by $\mathcal{L}^{-1/2}$). Let $\mathcal{R}=\nabla\mathcal{L}^{-1/2}$ be the Riesz transform associated with $\mathcal{L}$. Then $\mathcal{R}$ is bounded from $L^2(\R^3)$ to $L^2(\R^3;\R^3)$. Denote by $\mathcal{R}^{*}$ the adjoint of $\mathcal{R}$. Now for any $f\in L^2(\R^3;\R^3)$, one can write the gradient of the solution to \eqref{2nd order elliptic equation of divergence form} as $\nabla P=\mathcal{R}\mathcal{R}^{*}f$. We refer the reader to the monograph \cite{auscher book 1998} for more details in this paragraph.

The $L^p$ theory for the square root problem is based on the famous Aronson-Nash estimates for the kernel of $e^{-t\mathcal{L}}$.

\begin{lemma}[see, e.g., \cite{auscher jlms 1996}]\label{gauss property}
The semigroup $e^{-t\mathcal{L}}$ acting on $L^2$ has a kernel $K_t(x,y)$ satisfying the Gaussian property. Precisely, there exist constants $\mu\in(0,1]$ and $C>0$ depending only on $m$ such that for all $t>0$ and $x,y\in\R^3$,
\begin{align*}
|K_t(x,y)|\le Ct^{-3/2}\exp\left\{-\frac{|x-y|^2}{Ct}\right\},
\end{align*}
and if in addition $2|h|\le\sqrt{t}+|x-y|$,
\begin{align*}
|K_t(x+h,y)-K_t(x,y)|+|K_t(x,y+h)-K_t(x,y)|\\
\le Ct^{-3/2}\left(\frac{|h|}{\sqrt{t}+|x-y|}\right)^{\mu}\exp\left\{-\frac{|x-y|^2}{Ct}\right\}.
\end{align*}
\end{lemma}

Thanks to this property, $e^{-t\mathcal{L}}$ extrapolates to a bounded analytic semigroup on $L^p$ for every $p\in(1,\infty)$ (see \cite{ouhabaz pams 1995} and the references therein). We denote by $e^{-t\mathcal{L}_p}$ this semigroup and by $-\mathcal{L}_p$ the generator. Thus, the square root $\mathcal{L}_p^{1/2}$ of $\mathcal{L}_p$ on $L^p(\R^3)$ is also well-defined. The following result can be found in {\cite[pp. 131-132]{auscher book 1998}}.
\begin{lemma}
$\mathcal{L}^{1/2}$ extends to a continuous operator from $\dot{W}^{1,p}(\R^3)$ to $L^p(\R^3)$, $1<p<\infty$, which has a continuous inverse for $1<p<2+\varepsilon$ for some $\varepsilon>0$. Moreover, the extension of $\mathcal{L}^{1/2}$ on $W^{1,p}(\R^3)$ coincides with $\mathcal{L}_p^{1/2}$ for $1<p<2+\varepsilon$.
\end{lemma}
From now on, we do not distinguish between an operator and its continuous extension, and drop the $p$'s for all operators associated with $\mathcal{L}_p$. The above lemma implies that the Riesz transform $\mathcal{R}$ is bounded from $L^{p}(\R^3)$ to $L^{p}(\R^3;\R^3)$ for $1<p<2+\varepsilon$, and its adjoint $\mathcal{R}^{*}$ is bounded from $L^{p'}(\R^3;\R^3)$ to $L^{p'}(\R^3)$. Moreover, all the bounds (i.e., operator norms) only depend on $m$ and $p$.

\subsection{Homogeneous Besov spaces}

We collect some homogeneous Besov spaces relevant preliminaries. While there is a vast literature on this topic, we will mainly refer to the book \cite{BCD} because it focuses on applications of Fourier analysis to PDEs.

Let $\chi,\varphi$ be two smooth radial functions valued in the interval [0,1],
the support of $\chi$ be the ball $\overline{B_{4/3}}$, and the support of $\varphi$ be the washer $\overline{B_{8/3}}\setminus B_{3/4}$. Moreover, $\sum_{j\in\mathbb{Z}}\varphi(2^{-j}\xi)=1$ for $\xi\in\R^3\setminus\{0\}$, and $\chi(\xi)+\sum_{j\geq 0}\varphi(2^{-j}\xi)=1$ for $\xi\in\R^3$.
Denote $h=\mathcal{F}^{-1}\varphi$ and $\widetilde{h}=\mathcal{F}^{-1}\chi$, where $\mathcal{F}^{-1}$ is the inverse of the Fourier transform $\mathcal{F}$. Then we define the homogeneous dyadic blocks $\dot{\Delta}_{j}$ and the low-frequency cutoff operators $\dot{S}_{j}$, respectively, by
\begin{align*}
\dot{\Delta}_{j}u=2^{3j}\int_{\mathbb{R}^3}h(2^{j}y)u(x-y)dy,\ \ \mathrm{and}\ \ 
\dot{S}_{j}u=2^{3j}\int_{\mathbb{R}^3}\widetilde{h}(2^{j}y)u(x-y)dy.
\end{align*}

The frequency localization of the blocks $\dot{\Delta}_{j}$ and $\dot{S}_{j}$ leads to some fine properties. First, one has orthogonality due to the interaction of frequencies supported in disjoint regions. More precisely, for two tempered distributions $u,v\in\mathscr{S}^{'}(\R^3)$, we have
\begin{align*}
\dot{\Delta}_{k}\dot{\Delta}_{j}u=&0\ \ \mathrm{if}\ \ |k-j|\geq2,\\
\dot{\Delta}_{j}(\dot{S}_{k-1}u\dot{\Delta}_{k}v)=&0\ \ \mathrm{if}\ \ |k-j|\geq5,\\
\dot{\Delta}_{j}(\dot{\Delta}_{k}u\widetilde{\dot{\Delta}}_{k}v)=&0\ \ \mathrm{if}\ \ k\le j-4,\ \ \mathrm{with}\ \ \widetilde{\dot{\Delta}}_{k}v\vcentcolon=\sum_{|k^{'}-k|\leq1}\dot{\Delta}_{k^{'}}v. 
\end{align*}
Second, we have the Poincare type inequalities (also called Bernstein's inequalities):
\begin{lemma}[see {\cite[Lemmas 2.1-2.2]{BCD}}]\label{Bernstein Lemma}
Let $r$ and $R$ be two constants satisfying $0<r<R<\infty$.
There exists a positive constant $C=C(r,R)$ such that for any $k\in\mathbb{N}$,
any $\lambda>0$, any smooth homogeneous function $\sigma$ of degree $d\in\mathbb{N}$, any $1\leq p\leq q\leq\infty$, and any function $u\in L^{p}(\R^3)$,
\begin{align*}
\supp\widehat{u}\subset\lambda B_R\Longrightarrow&
\|D^{k}u\|_{q}\vcentcolon=\sum_{|\alpha|=k}\|\partial^{\alpha}u\|_{q}\le C^{k+1}\lambda^{k+3(\frac{1}{p}-\frac{1}{q})}\|u\|_{p},\\
\supp\widehat{u}\subset\lambda(B_R\setminus B_r)\Longrightarrow&
C^{-k-1}\lambda^{k}\|u\|_{p}\le\|D^{k}u\|_{p}\le C^{k+1}\lambda^{k}\|u\|_{p},\\
\supp\widehat{u}\subset\lambda(B_R\setminus B_r)\Longrightarrow&\|\sigma(D)u\|_{q}\le C_{\sigma,d}\lambda^{m+3(\frac1p-\frac1q)}\|u\|_{p},
\end{align*}
where $\widehat{u}=\mathcal{F}u$ and $\sigma(D)u$ is defined as $\mathcal{F}^{-1}(\sigma\widehat{u})$. 
\end{lemma}

In most literature on the theory of function spaces, the homogeneous Besov spaces are defined in the ambient space of tempered distributions modulo polynomials (see, e.g., \cite{Triebel book 1983}). However, we wish to avoid this type of spaces when solving nonlinear PDEs. Instead, we adopt the definitions of homogeneous spaces in {\cite[Section~2.3]{BCD}}. Let $\mathscr{S}_{h}^{'}(\R^3)$ denote the space of all tempered distributions $u\in\mathscr{S}^{'}(\R^3)$ that satisfy
\begin{align*}
u=\sum_{ j\in \mathbb{Z}}\dot{\Delta }_{j}u\ \  \mathrm{in}\  \mathscr{S}^{'}.
\end{align*}
Note that $\mathscr{S}_h^{'}$ is a large enough proper subspace of $\mathscr{S}^{'}$. For example, if $b\in L^p$ with $p\in[1,\infty)$, then $b\in\mathscr{S}_h^{'}$; and if $b\in L^\infty$, then $\nabla b\in\mathscr{S}_h^{'}$.


Let us now recall the definition of homogeneous Besov spaces.
\begin{definition}
Let $s\in\R$ and $1\leq p,r\leq\infty$. The homogeneous Besov space $\dot{B}_{p,r}^{s}(\R^3)$
consists of all the distributions $u$ in $\mathscr{S}_{h}^{'}(\R^3)$ such that
$$
\|u\|_{\dot{B}_{p,r}^{s}}\vcentcolon=\left\|\big(2^{js}\|\dot{\Delta}_{j}u\|_{L^{p}}\big)_{j\in\mathbb{Z}}\right\|_{l^{r}}
<\infty.
$$
\end{definition}
As an immediate consequence of the definition, we have that $u\in\dot{B}_{p,r}^{s}(\mathbb{R}^3)$ if and only if there exists a nonnegative sequence $\{c_{j,r}\}_{j\in\mathbb{Z}}$ such that $\|c_{j,r}\|_{l^{r}}\lesssim1$ and $\|\dot{\Delta}_{j}u\|_{L^p}\lesssim c_{j,r}2^{-js}\|u\|_{\dot{B}_{p,r}^{s}}$ for every $j\in\mathbb{Z}$.

Let us collect some useful properties and inequalities for homogeneous Besov spaces.

\begin{lemma}[see {\cite[Chapter~2]{BCD}}]\label{basic properties of Besov spaces}
{\rm (\romannumeral1)} For $1\le p_1\le p_2\le\infty$, $1\le r_1\le r_2\le\infty$ and $s\in\R$, $\dot{B}_{p_1,r_1}^{s}(\R^3)$ is continuously embedded in $\dot{B}_{p_2,r_2}^{s-3\left(1/p_1-1/p_2\right)}(\R^3)$. $\dot{B}_{\infty,1}^{0}(\R^3)$ is continuously embedded in $C_0(\R^3)$, the space of continuous functions that tend to zero at infinity.

{\rm (\romannumeral2)} Suppose that $(s_1,s_2,p,p_1,p_2,r)\in\R^2\times[1,\infty]^4$, $\theta\in(0,1)$ and $\frac{1}{p}=\frac{\theta}{p_1}+\frac{1-\theta}{p_2}$. Then for any $u\in\mathscr{S}^{'}_h(\R^3)$, we have
\begin{align*}
\|u\|_{\dot{B}_{p,r}^{\theta s_1+(1-\theta)s_2}}\le\|u\|_{\dot{B}_{p_1,r}^{s_1}}^{\theta}\|u\|_{\dot{B}_{p_2,r}^{s_2}}^{1-\theta}.
\end{align*}

{\rm (\romannumeral3)} If $(p,r)\in[1,\infty)^2$, then the space $\mathscr{S}_0(\R^3)$ of functions in $\mathscr{S}(\R^3)$ whose Fourier transforms are supported away from $0$ is dense in $\dot{B}_{p,r}^{s}(\R^3)$.

{\rm (\romannumeral4)} $\dot{B}_{p,r}^{s}(\R^3)$ is a Banach space if $(s,p,r)$ satisfies
\begin{align}\label{spr restriction}
s<\frac{3}{p},\ \ or\ \  s=\frac{3}{p}\ \ and\ \ r=1.
\end{align}

{\rm (\romannumeral5)} Suppose that  $(s,p,r)\in(0,\infty)\times[1,\infty]^2$ satisfies \eqref{spr restriction}. Assume that $f$ is a smooth function on $\R$ which vanishes at $0$. For any real-valued function $u\in L^\infty(\R^3)\cap\dot{B}_{p,r}^{s}(\R^3)$, the composite function $f\circ u$ belongs to the same space, and there exists a constant $C$ depending on $f'$ and $\|u\|_{\infty}$ such that
\begin{align*}
\|f\circ u\|_{\dot{B}_{p,r}^{s}}\le C\|u\|_{\dot{B}_{p,r}^{s}}. 
\end{align*}
\end{lemma}
Next, we recall Bony's paraproduct decomposition which can be used to define a product of two Besov functions. For two Besov functions $u$ and $v$, we can formally write
\begin{align*}
uv=\dot{T}_{u}v+\dot{T}_{v}u+\dot{R}(u,v),
\end{align*}
where
\begin{align*}
\dot{T}_{u}v\vcentcolon=\sum_{j\in\mathbb{Z}}\dot{S}_{j-1}u\dot{\Delta}_{j}v\ \ \ \mathrm{and}\ \ \ \dot{R}(u,v)\vcentcolon=\sum_{j\in\mathbb{Z}}\dot{\Delta}_{j}u\widetilde{\dot{\Delta}}_{j}v.
\end{align*}
Sometimes it is sufficient to just estimate $\dot{T}_{v}^{'}u\vcentcolon=\dot{T}_{v}u+\dot{R}(u,v)$.
We refer the reader to {\cite[section~2.6]{BCD}} for the convergence of the above series and the continuity of the paraproduct operators on homogeneous  Besov spaces. In the present paper, we will frequently use the following product laws:
\begin{align*}
\|uv\|_{\dot{B}_{p,1}^{3/p}}\lesssim&\|u\|_{\dot{B}_{p,1}^{3/p}}\|v\|_{\dot{B}_{p,1}^{3/p}},\ \ 1\le p<\infty,\\
\|uv\|_{\dot{B}_{2,1}^{1/2}}\lesssim&\|u\|_{\dot{B}_{q,1}^{3/q}}\|v\|_{\dot{B}_{2,1}^{1/2}},\ \ 1\le q\le6.
\end{align*}

Finally, we recall several equivalent characterizations of homogeneous Besov norms. The next lemma is well-known and can be implied by {\cite[Theorem~2.34]{BCD}}.
\begin{lemma}\label{characterization via classic heat kernel}
Suppose that $s\in\R$ and $(p,q)\in[1,\infty]^2$. If $k>s/2$ and $k\ge0$, we have
\begin{align*}
\|f\|_{\dot{B}_{p,q}^{s,-\Delta}}\vcentcolon=\left\|t^{-s/2}\|(t\Delta)^ke^{t\Delta}f\|_p\right\|_{L^q(\R_+;\frac{dt}{t})}\simeq\|f\|_{\dot{B}_{p,q}^{s}},\ \ \forall f\in\mathscr{S}_{h}^{'}.
\end{align*}
\end{lemma}

\begin{lemma}\label{characterization of Besov via ball mean difference}
Suppose that $s\in(0,1)$, $p\in[1,\infty)$ and $q\in[1,\infty]$. It holds that
\begin{align*}
\|f\|_{\dot{\Lambda}_{p,q}^{s}}\vcentcolon=\left\|\left(\int_{\R^3}\fint_{B(x,r)}\frac{|f(x)-f(y)|^p}{r^{sp}}\,dy\,dx\right)^{\frac1p}\right\|_{L^q(\R_+;\frac{dr}{r})}\simeq\|f\|_{\dot{B}_{p,q}^{s}},\ \ \forall f\in L^p,
\end{align*}
where $\fint_{B(x,r)}$ denotes the integral mean over the ball $B(x,r)$ centered at $x$ with radius $r$. 
\end{lemma}

\begin{lemma}[see \cite{bui adv 2012,liu fm 2015}]\label{characterization via heat kernel for divergence operator}
Let $\mu$ be the H\"{o}lder index in Lemma \ref{gauss property}. Suppose that $s\in(0,\mu)$, $p\in[1,\infty)$ and $q\in[1,\infty]$. If $k>s/2$, we have
\begin{align*}
\|f\|_{\dot{B}_{p,q}^{s,\mathcal{L}}}\vcentcolon=\left\|t^{-s/2}\|(t\mathcal{L})^ke^{-t\mathcal{L}}f\|_p\right\|_{L^q\left(\R_+;\frac{dt}{t}\right)}\simeq\|f\|_{\dot{B}_{p,q}^{s}},\ \ \ \forall f\in L^p.
\end{align*}
\end{lemma}

\begin{remark}
Let us give a remark on the proof of Lemma \ref{characterization of Besov via ball mean difference} and Lemma \ref{characterization via heat kernel for divergence operator}. In \cite{liu fm 2015}, only the characterization of inhomogeneous norms was given. But the proof there can be easily adjusted to give the equivalence between homogeneous norms. In fact, using the method in \cite{liu fm 2015}, one can obtain equivalence between $\|\cdot\|_{\dot{B}_{p,q}^{s,\mathcal{L}}}$ and $\|\cdot\|_{\dot{\Lambda}_{p,q}^{s}}$ instead of what is stated in Lemma \ref{characterization via heat kernel for divergence operator}. This together with Lemma \ref{characterization via classic heat kernel} implies Lemma \ref{characterization of Besov via ball mean difference} (note that $\mu=1$ if $\mathcal{L}$ is replaced by $-\Delta$). Hence, Lemma \ref{characterization via heat kernel for divergence operator} also follows. 
\end{remark}

\bigskip

\section{Maximal $L^1$ regularity for Stokes system}\label{main section}

\bigskip

This section is devoted to the proof of Theorem \ref{maximal regularity for Stokes}, which is the heart of the present paper.

\subsection{Stokes operator}
We first present some useful properties for the Stokes operator. Let $\rho=\rho(x)$ satisfy \eqref{initial density bounds} and denote $b=\rho^{-1}$. In the sequel, we use the notations $L^2=L^2(\R^3;\R^3)$, $\mathcal{P}L^2=\{u\in L^2|\divg u=0\}$, $\|\cdot\|$ the $L^2$ norm induced by the standard $L^2$ inner product $\langle\cdot,\cdot\rangle$, and $\|\cdot\|_{\rho}$ the weighted norm induced by the inner product
\begin{align*}
\langle u,v\rangle_{\rho}=\int_{\R^3}u(x)\cdot v(x)\rho(x)\,dx.
\end{align*}
Let $\mathcal{Q}_b$, $\mathcal{P}_b$ and $\mathcal{S}$ be defined by \eqref{Hodge b operator}, \eqref{Pb operator} and \eqref{Stokes operator}, respectively. It is well know that $\mathcal{Q}_b$ is bounded on $L^2$ and $\|\mathcal{Q}_bf\|_b\le\|f\|_b$ for every $f\in L^2$. If $b\equiv1$, we denote $\mathcal{P}=\mathcal{P}_1$ and $\mathcal{Q}=\mathcal{Q}_1$. Formally, we have $\divg\mathcal{P}=0$. It is for this reason we use $\mathcal{P}$ in front of a space of vectors to denote its subspace of divergence-free vectors.

\begin{lemma}\label{basic properties of P and S}
With the above notations, we have

{\rm (\romannumeral1)} $b\mathcal{P}_b:L^2\rightarrow \mathcal{P}L^2$ is bounded and it holds that $\mathcal{Q}(b\mathcal{P}_b)=b\mathcal{P}_b\mathcal{Q}\equiv0$, and $\mathcal{P}(b\mathcal{P}_b)=b\mathcal{P}_b\mathcal{P}\equiv b\mathcal{P}_b$.

{\rm (\romannumeral2)} $b\mathcal{P}_b:\mathcal{P}L^2\rightarrow \mathcal{P}L^2$ is invertible with a continuous inverse $\mathcal{P}\rho$. Thus, it holds that
\begin{align*}
\|b\mathcal{P}_bu\|\simeq\|u\|,\ \   \forall u\in\mathcal{P}L^2. 
\end{align*}

{\rm (\romannumeral3)} $b\mathcal{P}_b:L^2\rightarrow L^2$ is self-adjoint on $(L^2,\langle\cdot,\cdot\rangle)$, while $b\mathcal{P}_b:\mathcal{P}L^2\rightarrow \mathcal{P}L^2$ is self-adjoint on both $(\mathcal{P}L^2,\langle\cdot,\cdot\rangle)$ and $(\mathcal{P}L^2,\langle\cdot,\cdot\rangle_{\rho})$.

{\rm (\romannumeral4)} $\mathcal{S}:\mathcal{P}H^2\subset \mathcal{P}L^2\rightarrow \mathcal{P}L^2$ is a self-adjoint operator on $(\mathcal{P}L^2,\langle\cdot,\cdot\rangle_{\rho})$.
\end{lemma}
The proof is straightforward and thus left to the reader.

\begin{lemma}\label{Stokes operator generates bounded analytic semigroup on PL2}
The Stokes operator $\mathcal{S}:\mathcal{P}H^2\subset \mathcal{P}L^2\rightarrow \mathcal{P}L^2$ generates an analytic semigroup of contraction $\{e^{t\mathcal{S}}\}_{t\ge0}$ on $(\mathcal{P}L^2,\langle\cdot,\cdot\rangle_\rho)$, and $e^{t\mathcal{S}}b\mathcal{P}_b$ is self-adjoint on $(L^2,\langle\cdot,\cdot\rangle)$ for every $t\ge0$.
\end{lemma}
\begin{proof}
For $u\in \mathcal{P}H^2$, we have
\begin{align*}
\langle \mathcal{S} u,u\rangle_{\rho}=\int_{\R^n}\mathcal{P}_b\Delta u(x)\cdot u(x)\,dx=-\|\nabla u\|^2\le0.
\end{align*}
Since $\mathcal{S}$ is self-adjoint on $(\mathcal{P}L^2,\langle\cdot,\cdot\rangle_{\rho})$, so by Theorem \ref{generation theorem of selfadjoint operator}, $\mathcal{S}$ generates an analytic semigroup of contraction $\{e^{t\mathcal{S}}\}_{t\ge0}$ on $(\mathcal{P}L^2,\langle\cdot,\cdot\rangle_\rho)$. Moreover, $e^{t\mathcal{S}}$ is self-adjoint on $(\mathcal{P}L^2,\langle\cdot,\cdot\rangle_\rho)$. Then we have for all $u,v\in L^2$ that
\begin{align*}
\langle e^{t\mathcal{S}}b\mathcal{P}_bu,v\rangle=\langle e^{t\mathcal{S}}b\mathcal{P}_bu,b\mathcal{P}_bv\rangle_{\rho}=\langle b\mathcal{P}_bu,e^{t\mathcal{S}}b\mathcal{P}_bv\rangle_{\rho}=\langle u,e^{t\mathcal{S}}b\mathcal{P}_bv\rangle.
\end{align*}
This means that $e^{t\mathcal{S}}b\mathcal{P}_b$ is self-adjoint on $(L^2,\langle\cdot,\cdot\rangle)$.
\end{proof}

\begin{proposition}\label{asymptotics of semigroup}
For any $u_0\in \mathcal{P}L^2$, it holds that $\lim_{t\rightarrow\infty}\|e^{t\mathcal{S}}u_0\|=0$.
\end{proposition}
\begin{proof}
In view of Lemma \ref{Stokes operator generates bounded analytic semigroup on PL2} and Lemma \ref{basic properties of P and S} {\rm (\romannumeral2)}, we apply Gagliardo–Nirenberg inequality to get
\begin{align*}
\|e^{t\mathcal{S}}u_0\|_{p}\le C\|e^{t\mathcal{S}}u_0\|^{1-\theta} \|\mathcal{S}e^{t\mathcal{S}}u_0\|^{\theta}\le Ct^{-\theta}\|u_0\|,\ \forall u_0\in\mathcal{P}L^2,
\end{align*}
where $\frac{1}{p}=\frac12-\frac{2\theta}{3}$, $p\in(2,\infty]$ and $\theta\in(0,1)$.
Now for any $u_0\in L^2$, since $b\mathcal{P}_bu_0\in \mathcal{P}L^2$, we get $\|e^{t\mathcal{S}}b\mathcal{P}_bu_0\|_{p}\le Ct^{-\theta}\|u_0\|$.
By duality, we obtain $\|e^{t\mathcal{S}}b\mathcal{P}_bu_0\|\le Ct^{-\theta}\|u_0\|_{p'}$ with $\frac{1}{p'}=\frac12+\frac{2\theta}{3}$. So for any $u_0\in \mathcal{P}L^2\cap L^{p'}$, we have 
\begin{align*}
\|e^{t\mathcal{S}}u_0\|=\|e^{t\mathcal{S}}b\mathcal{P}_b(\rho u_0)\|\le Ct^{-\theta}\|u_0\|_{p'}, 
\end{align*}
which implies $\lim_{t\rightarrow\infty}\|e^{t\mathcal{S}}u_0\|=0$. By a density argument, the result still holds for $u_0\in \mathcal{P}L^2$. This completes the proof.
\end{proof}

\subsection{Characterizations of Besov norms via semigroup $e^{t\mathcal{S}}$}
In order to obtain maximal regularity estimates for the Stokes system, we use the semigroup $e^{t\mathcal{S}}$ to give equivalent characterizations of certain Besov norms for divergence-free vector fields. First, let us prove an easy but useful lemma.
\begin{lemma}\label{weighted estimates}
Let $w(t,\tau)$ be a nonnegative weight function satisfying
\begin{align*}
\sup_{\tau>0}\int_0^\infty w(t,\tau)\,\frac{dt}{t}+\sup_{t>0}\int_0^\infty w(t,\tau)\,\frac{d\tau}{\tau}\le C.
\end{align*}
Then for any $q\in[1,\infty]$, we have
\begin{align*}
\left\|\int_0^\infty w(t,\tau)f(\tau)\frac{d\tau}{\tau}\right\|_{L^q(\R_+,\frac{dt}{t})}\le C\|f\|_{L^q(\R_+,\frac{dt}{t})}.
\end{align*}
\end{lemma}
\begin{proof}
The cases $q=1$ and $q=\infty$ are trivial. Assume that $f\ge0$. For $q\in(1,\infty)$, we apply H\"{o}lder's inequality to see that
\begin{align*}
\int_0^\infty w(t,\tau)f(\tau)\frac{d\tau}{\tau}\le\left(\int_0^\infty w(t,\tau)\frac{d\tau}{\tau}\right)^{1/q'} \left(\int_0^\infty w(t,\tau)f^q(\tau)\frac{d\tau}{\tau}\right)^{1/q}.
\end{align*}
By the assumption of the lemma, we obtain
\begin{align*}
\left(\int_0^\infty w(t,\tau)f(\tau)\frac{d\tau}{\tau}\right)^{q}\le C\int_0^\infty w(t,\tau)f^q(\tau)\frac{d\tau}{\tau}.
\end{align*}
Integrating both sides over $(0,\infty)$ with respect to $\frac{dt}{t}$, then the result follows from the assumption again and the Fubini's theorem.
\end{proof}

As in many other works (see, e.g., \cite{bui adv 2012,liu fm 2015}) concerning characterizations of function spaces via semigroups, a fundamental ingredient is to obtain a sort of reproducing formulas associated with the semigroups. The following reproducing formula depends on the very special (and simple) structure of the operator $\mathcal{S}$.
\begin{lemma}
For any $u_0\in \mathcal{P}L^2$, it holds that
\begin{align}\label{reproducing formula}
u_0=-\int_0^\infty\Delta e^{\tau \mathcal{S}}b\mathcal{P}_bu_0\,d\tau\vcentcolon=\lim_{\varepsilon\rightarrow0^+}-\int_{\varepsilon}^{1/\varepsilon}\Delta e^{\tau \mathcal{S}}b\mathcal{P}_bu_0\,d\tau,
\end{align}
where the limit converges in $L^2$.
\end{lemma}
\begin{proof}
Since $e^{t\mathcal{S}}$ is an analytic semigroup, the function $u(t)=e^{t\mathcal{S}}u_0$ is an classical solution to the equation $u'(t)=\mathcal{S} u(t)$. Integrating this equation in time from $s$ to $t$, we get
\begin{align*}
u(t)-u(s)=\int_s^t\mathcal{S} u(\tau)\,d\tau.
\end{align*}
Obviously, $u(s)\rightarrow u_0$ in $L^2$ as $s\rightarrow0^+$.
This together with Proposition \ref{asymptotics of semigroup} implies that
\begin{align*}
u_0=-\int_0^\infty \mathcal{S} e^{\tau \mathcal{S}}u_0\,d\tau.
\end{align*}
Replacing $u_0$ by $b\mathcal{P}_bu_0$ and recalling the expression for $\mathcal{S}$, we have
\begin{align*}
b\mathcal{P}_bu_0=-\int_0^\infty b\mathcal{P}_b\Delta e^{\tau \mathcal{S}}b\mathcal{P}_bu_0\,d\tau.
\end{align*}
Then the desired formula follows from the fact that $b\mathcal{P}_b$ is invertible on $\mathcal{P}L^2$.
\end{proof}

Let us first consider characterizations of Besov norms with negative regularity.

\begin{theorem}\label{characterization of besov spaces with negative regularity via stokes}
Suppose that $s\in(0,2)$ and $q\in[1,\infty]$. For any $u_0\in \mathcal{P}L^2$, we have
\begin{align*}
\left\|t^{s/2}\|e^{t\mathcal{S}}b\mathcal{P}_bu_0\|\right\|_{L^q(\R_+,\frac{dt}{t})}\simeq\|u_0\|_{\dot{B}_{2,q}^{-s}}.
\end{align*}
\end{theorem}
\begin{proof}
Let us first assume that $u_0\in\dot{B}_{2,q}^{-s}$. We need the reproducing formula (by taking $b\equiv1$ in \eqref{reproducing formula})
\begin{align*}
u_0=-\int_0^\infty\Delta e^{\tau\Delta}u_0\,d\tau,\ \ \  u_0\in \mathcal{P}L^2.
\end{align*}
Next, applying $e^{t\mathcal{S}}b\mathcal{P}_b$ to both sides of the above equation gives rise to
\begin{align*}
e^{t\mathcal{S}}b\mathcal{P}_bu_0=-\int_0^\infty e^{t\mathcal{S}}b\mathcal{P}_b\Delta e^{\tau\Delta}u_0\,d\tau.
\end{align*}
We may estimate the $L^2$ norm of the integrand in two different ways: we get from Lemma \ref{Stokes operator generates bounded analytic semigroup on PL2} that
\begin{align*}
\|e^{t\mathcal{S}}b\mathcal{P}_b\Delta e^{\tau\Delta}u_0\|=\|\mathcal{S}e^{t\mathcal{S}}e^{\tau\Delta}u_0\|\le\frac{C}{t}\|e^{\tau\Delta}u_0\|\le\frac{C}{t}\|e^{\frac{\tau}{2}\Delta}u_0\|,
\end{align*}
alternatively,
\begin{align*}
\|e^{t\mathcal{S}}b\mathcal{P}_b\Delta e^{\tau\Delta}u_0\|\le C\|\Delta e^{\tau\Delta}u_0\|\le\frac{C}{\tau}\|e^{\frac{\tau}{2}\Delta}u_0\|.
\end{align*}
Hence, we get
\begin{align*}
\|e^{t\mathcal{S}}b\mathcal{P}_bu_0\|\le C\int_0^\infty\left(\frac1t\wedge\frac{1}{\tau}\right) \|e^{\tau\Delta}u_0\|\,d\tau.
\end{align*}
Multiplying both sides by $t^{s/2}$, we write
\begin{align*}
t^{s/2}\|e^{t\mathcal{S}}b\mathcal{P}_bu_0\|\le C\int_0^\infty\left(\frac{t}{\tau}\right)^{s/2}\left(\frac{\tau}{t}\wedge1\right) \|\tau^{s/2}e^{\tau\Delta}u_0\|\,\frac{d\tau}{\tau}.
\end{align*}
Now if $s\in(0,2)$, it is easy to verify that
\begin{align*}
\sup_{t>0}\int_0^\infty\left(\frac{t}{\tau}\right)^{s/2}\left(\frac{\tau}{t}\wedge1\right)\,\frac{d\tau}{\tau}+\sup_{\tau>0}\int_0^\infty\left(\frac{t}{\tau}\right)^{s/2}\left(\frac{\tau}{t}\wedge1\right)\,\frac{dt}{t}
\le C.
\end{align*}
It follows form Lemma \ref{weighted estimates} and Lemma \ref{characterization via classic heat kernel} that
\begin{align*}
\left\|t^{s/2}\|e^{t\mathcal{S}}b\mathcal{P}_bu_0\|\right\|_{L^q(\R_+,\frac{dt}{t})}\le C\left\|t^{s/2}\|e^{t\Delta}u_0\|\right\|_{L^q(\R_+,\frac{dt}{t})}\le C\|u_0\|_{\dot{B}_{2,q}^{-s}}.
\end{align*}

For the reverse inequality, we apply $e^{t\Delta}$ to both sides of \eqref{reproducing formula} to get
\begin{align*}
e^{t\Delta}u_0=-\int_0^\infty e^{t\Delta}\Delta e^{\tau \mathcal{S}}b\mathcal{P}_bu_0\,d\tau.
\end{align*}
Again, in view of Lemma \ref{Stokes operator generates bounded analytic semigroup on PL2}, a similar argument as before gives rise to
\begin{align*}
\|e^{t\Delta}u_0\|\le C\int_0^\infty \left(\frac1t\wedge\frac{1}{\tau}\right)\|e^{\tau \mathcal{S}}b\mathcal{P}_bu_0\|\,d\tau.
\end{align*}
The rest of the steps are exactly the same as before. So the proof of the theorem is completed.
\end{proof}

Now Theorem \ref{characterization of besov spaces with positive regularity via stokes} is an immediate consequence of Theorem \ref{characterization of besov spaces with negative regularity via stokes}.

\begin{proof}[Proof of Theorem \ref{characterization of besov spaces with positive regularity via stokes}]
Since $u_0\in \mathcal{P}H^2$, the result immediately follows from Theorem \ref{characterization of besov spaces with negative regularity via stokes} if we replace $u_0$ by $\Delta u_0\in \mathcal{P}L^2$.
\end{proof}

With Theorem \ref{characterization of besov spaces with positive regularity via stokes} in hand, we can extrapolate $e^{t\mathcal{S}}$ to a semigroup on Besov spaces without assuming any regularity on the coefficient $b$. To this end, let us first study some regularity estimates of $e^{t\mathcal{S}}$ on $\mathcal{P}\dot{B}_{2,q}^{s}$.

\begin{proposition}
Suppose that $s\in(0,2)$, $q\in[1,\infty]$ and $k\in\mathbb{N}\cup\{0\}$. There exists a positive constant $C$ such that for any $u_0\in \mathcal{P}H^2$,
\begin{align}\label{analyticity in Besov spaces}
\|(t\mathcal{S})^ke^{t\mathcal{S}}u_0\|_{\dot{B}_{2,q}^{s}}\le C\|u_0\|_{\dot{B}_{2,q}^{s}},\ \ \forall t\ge0,
\end{align}
and
\begin{align}\label{space time estimates}
\left\|\|(t\mathcal{S})^{k+1} e^{t\mathcal{S}}u_0\|_{\dot{B}_{2,q}^{s}}\right\|_{L^q(\R_+,\frac{dt}{t})}\le C\|u_0\|_{\dot{B}_{2,q}^{s}}.
\end{align}
\end{proposition}
\begin{proof}
The first inequality follows immediately from Theorem \ref{characterization of besov spaces with positive regularity via stokes} and the fact that $e^{t\mathcal{S}}$ is a bounded analytic semigroup on $\mathcal{P}L^2$.

For the second inequality, we only need to prove for $k=0$ and $q<\infty$. Applying Theorem \ref{characterization of besov spaces with positive regularity via stokes}, Lemma \ref{Stokes operator generates bounded analytic semigroup on PL2} and Fubini's theorem, we have
\begin{align*}
\int_0^\infty\|\tau\mathcal{S}e^{\tau\mathcal{S}}u_0\|_{\dot{B}_{2,q}^{s}}^q\,\frac{d\tau}{\tau}\simeq&\int_0^\infty\int_{0}^{\infty}\left(t^{-s/2}\|t\tau \mathcal{S}^2e^{(t+\tau)\mathcal{S}}u_0\|\right)^q \,\frac{dt}{t}\,\frac{d\tau}{\tau}\\
=&\int_0^\infty\int_{\tau}^{\infty}(t-\tau)^{(1-s/2)q-1}\tau^{q-1}\|\mathcal{S}^2e^{t\mathcal{S}}u_0\|^q \,dt\,d\tau\\
=&\int_0^\infty\|\mathcal{S}^2e^{t\mathcal{S}}u_0\|^q\,dt\int_{0}^{t}(t-\tau)^{(1-s/2)q-1}\tau^{q-1}\,d\tau\\
\le&C\int_0^\infty\left(t^{-s/2}\|t\mathcal{S}e^{t\mathcal{S}}u_0\|\right)^q \,\frac{dt}{t}
\le C\|u_0\|_{\dot{B}_{2,q}^{s}}^q.
\end{align*}
This completes the proof.
\end{proof}

In the rest of this subsection, we assume that $s$ and $q$ satisfy
\begin{align}\label{s and q assumption}
(s,q)\in(0,3/2)\times[1,\infty), \ \ \mathrm{or} \  s\in(0,3/2]\ \ \mathrm{and}\ \ q=1.
\end{align}
Then $\mathcal{P}\dot{B}_{2,q}^{s}$ is a Banach space and $\mathcal{P}H^2$ is dense in $\mathcal{P}\dot{B}_{2,q}^{s}$. For each $t\ge0$, the inequality \eqref{analyticity in Besov spaces} (with $k=0$) guarantees that $e^{t\mathcal{S}}$ extends to a bounded operator on $\mathcal{P}\dot{B}_{2,q}^{s}$ with bounds uniform in $t$. We denote this extension by $\mathcal{T}(t)=\mathcal{T}_{s,q}(t)$. Then $\{\mathcal{T}(t)\}_{t\ge0}$ is a bounded semigroup on $\mathcal{P}\dot{B}_{2,q}^{s}$. Also, it is easy to verify the strong continuity for $\mathcal{T}(t)$.

\begin{proposition}
Suppose that $(s,q)$ satisfies \eqref{s and q assumption}. Then $\mathcal{T}(t)$ is a bounded $C_0$ semigroup on $\mathcal{P}\dot{B}_{2,q}^{s}$.
\end{proposition}
\begin{proof}
For $u_0\in\mathcal{P}H^2$, the function $t\mapsto\mathcal{T}(t)u_0=e^{t\mathcal{S}}u_0$ belongs to $C([0,\infty);\mathcal{P}H^2)$, thus $C([0,\infty);\mathcal{P}\dot{B}_{2,q}^{s})$. By a density argument, we get the strong continuity of $\mathcal{T}(t)$ on $\mathcal{P}\dot{B}_{2,q}^{s}$.
\end{proof}

Let us denote by $\mathcal{G}=\mathcal{G}_{s,q}$ the generator of $\mathcal{T}(t)$ on $\mathcal{P}\dot{B}_{2,q}^{s}$. In general, it is not easy to identify the domain of the generator of a semigroup. However, it would be easier to find a core for the generator.

\begin{lemma}\label{core for G}
Suppose that $(s,q)$ satisfies \eqref{s and q assumption}. Define $\mathscr{C}=\{u\in\mathcal{P}H^2|\mathcal{S}u\in\mathcal{P}\dot{B}_{2,q}^{s}\}$. Then $\mathscr{C}$ is a core for $\mathcal{G}$, and it holds that $\mathcal{G}|_{\mathscr{C}}=\mathcal{S}|_{\mathscr{C}}$. In other words, $\mathcal{G}$ is the closure of $\mathcal{S}:\mathscr{C}\subset\mathcal{P}\dot{B}_{2,q}^{s}\rightarrow\mathcal{P}\dot{B}_{2,q}^{s}$.
\end{lemma}
\begin{proof}
Note that the domain $D(\mathcal{S}^2)$ of $\mathcal{S}^2$ is contained in $\mathscr{C}$. Since $D(\mathcal{S}^2)$ is dense in $D(\mathcal{S})=\mathcal{P}H^2$ and $\mathcal{P}H^2$ is dense in $\mathcal{P}\dot{B}_{2,q}^{s}$, so $\mathscr{C}$ is also dense in $\mathcal{P}\dot{B}_{2,q}^{s}$. For every $u_0\in\mathscr{C}$, since $\mathcal{T}(t)u_0=e^{t\mathcal{S}}u_0$ is a classical solution to the equation $u'(t)=\mathcal{S}u(t)$, we have 
\begin{align*}
\mathcal{T}(t)u_0-u_0=\int_0^t e^{\tau \mathcal{S}}\mathcal{S}u_0\,d\tau=\int_0^t\mathcal{T}(\tau)\mathcal{S}u_0\,d\tau.
\end{align*}
So, we have
\begin{align*}
\frac1t (\mathcal{T}(t)u_0-u_0)=\frac1t\int_0^t\mathcal{T}(\tau)\mathcal{S}u_0\,d\tau.
\end{align*}
By strong continuity of $\mathcal{T}(t)$ on $\mathcal{P}\dot{B}_{2,q}^{s}$, the limit as $t\rightarrow0^+$ on the right exists and equals to $\mathcal{S}u_0$. We thus infer that $\mathscr{C}\subset D(\mathcal{G})$ and $\mathcal{G}|_{\mathscr{C}}=\mathcal{S}|_{\mathscr{C}}$. Obviously, $\mathscr{C}$ is invariant under $\mathcal{T}(t)$. Thus, by Lemma \ref{identify a core}, $\mathscr{C}$ is a core for $\mathcal{G}$. This completes the proof.
\end{proof}

\begin{proposition}\label{T is analytic}
Suppose that $(s,q)$ satisfies \eqref{s and q assumption}. Then $\mathcal{T}(t)$ is a bounded analytic semigroup on $\mathcal{P}\dot{B}_{2,q}^{s}$.
\end{proposition}
\begin{proof}
We know from the above lemma that $\mathscr{C}$ is dense in $\mathcal{P}\dot{B}_{2,q}^{s}$, and that $\mathcal{G}\mathcal{T}(t)u_0=\mathcal{S}e^{t\mathcal{S}}u_0$ for $u_0\in \mathscr{C}$. It then follows from \eqref{analyticity in Besov spaces} that $\|t\mathcal{G}\mathcal{T}(t)u_0\|_{\dot{B}_{2,q}^{s}}\le C\|u_0\|_{\dot{B}_{2,q}^{s}}$. So we apply Lemma \ref{generation theorem for analytic semigroups} to conclude that $\mathcal{T}(t)$ is an analytic semigroup.
\end{proof}

\begin{remark}
Now \eqref{space time estimates} actually holds for data in $\mathcal{P}\dot{B}_{2,q}^{s}$, that is,
\begin{align*}
\left\|\|(t\mathcal{G})e^{t\mathcal{G}}u_0\|_{\dot{B}_{2,q}^{s}}\right\|_{L^q(\R_+,\frac{dt}{t})}\le C\|u_0\|_{\dot{B}_{2,q}^{s}},\ \ \forall u_0\in\mathcal{P}\dot{B}_{2,q}^{s}.
\end{align*}
In particular, choosing $q=1$ gives
\begin{align}\label{space time estimate G}
\|\mathcal{G} e^{t\mathcal{G}}u_0\|_{L^1(\R_+;\dot{B}_{2,1}^{s})}\le C\|u_0\|_{\dot{B}_{2,1}^{s}},\ \ \forall u_0\in\mathcal{P}\dot{B}_{2,1}^{s}.
\end{align}
\end{remark}

We conclude this subsection with a maximal $L^1$ regularity result for the following abstract Cauchy problem
\begin{align}\label{ACP G}
u'(t)-\mathcal{G}u(t)=f(t),\ \ u(0)=u_0.
\end{align}

\begin{theorem}\label{maximal regularity for ACP rough b}
Suppose that $s\in(0,3/2]$ and $T\in(0,\infty]$. Let $u_0\in\mathcal{P}\dot{B}_{2,1}^{s}$ and $f\in L^1((0,T);\mathcal{P}\dot{B}_{2,1}^{s})$. Then \eqref{ACP G} has a unique strong solution $u\in C_b([0,T);\mathcal{P}\dot{B}_{2,1}^{s})$. Moreover, there exists a positive constant $C=C(m,s)$ such that
\begin{align*}
\|u\|_{L_T^\infty(\dot{B}_{2,1}^{s})}+\|u',\mathcal{G}u\|_{L_T^1(\dot{B}_{2,1}^{s})}\le C\|u_0\|_{\dot{B}_{2,1}^{s}}+C\|f\|_{L_T^1(\dot{B}_{2,1}^{s})}.
\end{align*}
\end{theorem}
\begin{proof}
The homogeneous part $e^{t\mathcal{G}}u_0$ is a classical solution, and satisfies the estimates by Proposition \ref{T is analytic} and \eqref{space time estimate G}. Denote the inhomogeneous part by $\mathcal{I}f(t)=\int_0^t e^{(t-\tau)\mathcal{G}}f(\tau)\,d\tau$. Since $e^{t\mathcal{G}}$ is uniformly bounded, we have $\|\mathcal{I}f\|_{L_T^\infty(\dot{B}_{2,1}^{s})}\le C\|f\|_{L_T^1(\dot{B}_{2,1}^{s})}$. Using again \eqref{space time estimate G} and Fubini's theorem, we have
\begin{align*}
\|\mathcal{G}\mathcal{I}f\|_{L_T^1(\dot{B}_{2,1}^{s})}\le&\int_0^T\int_0^t\|\mathcal{G}e^{(t-\tau)\mathcal{G}}f(\tau)\|_{\dot{B}_{2,1}^{s}}\,d\tau\,dt\\
=&\int_0^T\,d\tau\int_{\tau}^T\|\mathcal{G}e^{(t-\tau)\mathcal{G}}f(\tau)\|_{\dot{B}_{2,1}^{s}}\,dt\le C\|f\|_{L_T^1(\dot{B}_{2,1}^{s})}.
\end{align*}
So by Lemma \ref{mild to strong}, $u$ is a strong solution to \eqref{ACP G}. The estimate for $u'$ follows directly from the previous estimates and the equation \eqref{ACP G}. So the proof is completed.
\end{proof}

\subsection{Elliptic estimates}
So far we have not assumed any regularity on the coefficient $b$. In what follows, we shall prove that $b\mathcal{P}_b$ is bounded on some Besov spaces if $b$ has suitable "critical" regularity. We allow a slight discontinuity for $b$ and point out that it is of independent interest to study elliptic estimates with discontinuous coefficients. The continuity of $b\mathcal{P}_b$ will also help us identify the domain of $\mathcal{G}$.

In the sequel, $P\in\dot{H}^1(\R^3)$ is the weak solution to \eqref{2nd order elliptic equation of divergence form} with $b$ satisfying \eqref{initial density bounds}, and $\mu$ is the H\"{o}lder index in Lemma \ref{gauss property}. The main result in this subsection is the following:
\begin{theorem}\label{elliptic estimates in Besov}
Given any $(p,r)\in[2,\frac{3}{1-\mu})\times[1,\infty]$, any $s\in(0,\frac{3}{p}+\mu-1)$. If $b$ satisfies $\nabla b\in\dot{B}_{q,\infty}^{3/q-1}(\R^3)$ with $\frac{3}{q}>s\vee(1-\mu)$, there exists a constant $C$ depending on $m$ and $\|b\|_{\dot{B}_{q,\infty}^{3/q}}$ such that
\begin{align*}
\|\nabla P\|_{\dot{B}_{p,r}^{s}}\le C\|f\|_{\dot{B}_{p,r}^{s}}
\end{align*}
for all vectors $f$ whose components belong to $L^2(\R^3)\cap\dot{B}_{p,r}^{s}(\R^3)$.
\end{theorem}

The proof of Theorem \ref{elliptic estimates in Besov} takes two steps. In the first step, we do not assume any regularity on the elliptic coefficients and allow a loss of regularity for $\nabla P$.

\begin{lemma}\label{Initial iteration}
Given any $(p,r)\in[2,\frac{3}{1-\mu})\times[1,\infty]$, any $s\in(0,\frac{3}{p}+\mu-1)$. For any $-s_0\in(1-\mu,\frac{3}{p}-s]$, let $p_0$ be defined by $s_0-\frac{3}{p_0}=s-\frac{3}{p}$. Then there exists a constant $C>0$ such that
\begin{align*}
\|\nabla P\|_{\dot{B}_{p_0,r}^{s_0}}\le C\|f\|_{\dot{B}_{p,r}^{s}}, \ \ \forall f\in L^2(\R^3)\cap \dot{B}_{p,r}^{s}(\R^3).
\end{align*}
\end{lemma}
\begin{proof}
Note that the assumptions on $(p,r,s)$ guarantee the existence of $(p_0,s_0)$. By Lemma \ref{basic properties of Besov spaces} {\rm (\romannumeral1)}, we may assume that $s_0>-1$ and $p_0<\infty$. First, by Lemma \ref{characterization via heat kernel for divergence operator}, we have
\begin{align*}
\|\nabla P\|_{\dot{B}_{p_0,r}^{s_0}}\simeq&\|\mathcal{L}^{-1/2}\mathcal{R}^{*}f\|_{\dot{B}_{p_0,r}^{s_0+1}}\\
\simeq&\left\|t^{-\frac12-\frac{s_0}{2}}\|(t\mathcal{L})^{\frac12}e^{-t\mathcal{L}}\mathcal{L}^{-1/2}\mathcal{R}^{*}f\|_{p_0}\right\|_{L^r(\R_+,\frac{dt}{t})}\\
=&\left\|t^{-\frac{s_0}{2}}\|e^{-t\mathcal{L}}\mathcal{R}^{*}f\|_{p_0}\right\|_{L^r(\R_+,\frac{dt}{t})}.
\end{align*}
Next, applying $e^{-t\mathcal{L}}\mathcal{R}^{*}$ to the reproducing formula 
\begin{align*}
f(x)=\frac{1}{(k-1)!}\int_0^\infty(-\tau\Delta)^ke^{\tau\Delta}f(x)\,\frac{d\tau}{\tau}\ \ \mathrm{with}\ \ k>\frac{s}{2},
\end{align*}
we write
\begin{align*}
e^{-t\mathcal{L}}\mathcal{R}^{*}f(x)=
\frac{1}{(k-1)!}\int_0^\infty e^{-t\mathcal{L}}\mathcal{R}^{*}e^{\frac{\tau}{2}\Delta}(-\tau\Delta)^ke^{\frac{\tau}{2}\Delta}f(x)\,\frac{d\tau}{\tau}.
\end{align*}
The $L^p-L^q$ estimates for heat semigroups and the boundedness of $\mathcal{R}^*$ from $L^p(\R^3,\R^3)$ to $L^p(\R^3)$ ($2\le p<\infty$) then imply that
\begin{align*}
\|e^{-t\mathcal{L}}\mathcal{R}^{*}f\|_{p_0}\le C
\int_0^\infty\left(\frac1t\wedge\frac1\tau\right)^{\frac{3}{2}(\frac1p-\frac{1}{p_0})}\|(\tau\Delta)^ke^{\tau\Delta}f\|_p\,\frac{d\tau}{\tau}.
\end{align*}
Thus,
\begin{align*}
t^{-\frac{s_0}{2}}\|e^{-t\mathcal{L}}\mathcal{R}^{*}f\|_{p_0}\le C
\int_0^\infty \left(\frac{\tau}{t}\right)^{s_0/2}
\left(\frac{\tau}{t}\wedge1\right)^{\frac{s-s_0}{2}}\tau^{-s/2}\|(\tau\Delta)^ke^{\tau\Delta}f\|_p
\,\frac{d\tau}{\tau}.
\end{align*}
Note that $s$ and $-s_0$ are positive numbers. One can readily check that
\begin{align*}
\sup_{\tau>0}\int_0^\infty\left(\frac{\tau}{t}\right)^{s_0/2}\left(\frac{\tau}{t}\wedge1\right)^{\frac{s-s_0}{2}}\,\frac{dt}{t}+\sup_{t>0}\int_0^\infty\left(\frac{\tau}{t}\right)^{s_0/2}\left(\frac{\tau}{t}\wedge1\right)^{\frac{s-s_0}{2}}\,\frac{d\tau}{\tau}\le C.
\end{align*}
Finally, we apply Lemma \ref{weighted estimates} and Lemma \ref{characterization via classic heat kernel} to conclude the proof.
\end{proof}

In the second step, if the coefficient $b$ has suitable regularity, the loss of regularity of $\nabla P$ can be recovered using an iteration technique in the spirit of De Giorgi-Nash-Moser. To this end, we need the following commutator estimates that can help us gain regularity.

\begin{lemma}\label{commutator estimates}
Suppose that $r\in[1,\infty]$, $1\le p_2<p_1\le\infty$, $q\in[1,\infty)$, $(s_1,s_2)\in\R^2$, $s_1-\frac{3}{p_1}=s_2-\frac{3}{p_2}$, and
\begin{align*}
\frac{1}{p_2}\le\frac{1}{p_1}+\frac{1}{q},\  \frac{1}{p_2}<\frac{1}{p_1}+\frac{1}{3},\  \frac{3}{q\vee p_2}>s_2,\ \ \frac{3}{q\vee p'_{1}}>-s_1.
\end{align*}
If $a\in L^\infty(\R^3)$ and $\nabla a\in\dot{B}_{q,\infty}^{3/q-1}(\R^3)$, then there exists a constant $C$ such that
\begin{align*}
\left\|\left(2^{js_2}\|[\dot{\Delta}_j,a]f\|_{p_2}\right)_j\right\|_{l_j^r}\le C\|a\|_{\dot{B}_{q,\infty}^{3/q}}\|f\|_{\dot{B}_{p_1,r}^{s_1}},
\end{align*}
where $[\dot{\Delta}_j,a]f$ denotes the commutator $\dot{\Delta}_j(af)-a\dot{\Delta}_jf$. 
\end{lemma}
\begin{proof}
This type of estimate is nowadays classical (see, e.g., {\cite[Section~2.10]{BCD}}). We give proof here for the sake of completeness.

By Bony's paraproduct, we split the commutator $[\dot{\Delta}_j,a]f$ into four terms:
\begin{align}\label{split commutator}
[\dot{\Delta}_j,\dot{T}_a]f+\dot{\Delta}_j(\dot{T}_{f}a)+\dot{\Delta}_j(\dot{R}(a,f))-\dot{T}_{\dot{\Delta}_jf}^{'}a.
\end{align}
Many terms in the summation can be canceled out because of the frequency localization of the dyadic blocks. Specifically, the first term can be expressed as
\begin{align*}
[\dot{\Delta}_j,\dot{T}_a]f=\sum_{|j'-j|\le4}2^{nj}\int h(2^jy)\big(\dot{S}_{j'-1}a(x-y)-\dot{S}_{j'-1}a(x)\big)\dot{\Delta}_{j'}f(x-y)\,dy
\end{align*}
Choosing $p$ such that $\frac{1}{p_2}=\frac{1}{p}+\frac{1}{p_1}$, we use H\"{o}lder's inequality to see
\begin{align*}
\|[\dot{\Delta}_j,\dot{T}_a]f\|_{p_2}\lesssim& \sum_{|j'-j|\le4}2^{nj}\int |h(2^jy)|\|\dot{S}_{j'-1}a(\cdot-y)-\dot{S}_{j'-1}a(\cdot)\|_{p}\|\dot{\Delta}_{j'}f\|_{p_1}\,dy\\
\lesssim&2^{-j}\sum_{|j'-j|\le4}\|\nabla\dot{S}_{j'-1}a\|_{p}\|\dot{\Delta}_{j'}f\|_{p_1}.
\end{align*}
Noticing that $p\ge q$ and $1-\frac{3}{p}>0$, we use Lemma \ref{Bernstein Lemma} to get
\begin{align*}
\|\nabla\dot{S}_{j'-1}a\|_{p}\lesssim&\sum_{k\le j'-2}2^{3k(1/q-1/p)}\|\dot{\Delta}_{k}\nabla a\|_{q}\\\lesssim&\sum_{k\le j'-2}2^{k(1-3/p)}\|\nabla a\|_{\dot{B}_{q,\infty}^{3/q-1}}\lesssim2^{j'(1-3/p)}\|a\|_{\dot{B}_{q,\infty}^{3/q}}.
\end{align*}
Consequently,
\begin{align}\label{commutator estimate 1}
\|[\dot{\Delta}_j,\dot{T}_a]f\|_{p_2}\lesssim c_{j,r}2^{-js_2}\|a\|_{\dot{B}_{q,\infty}^{3/q}}\|f\|_{\dot{B}_{p_1,r}^{s_1}}.
\end{align}

For the second term in \eqref{split commutator}, if $q\ge p_2$, we redefine $p$ by $\frac{1}{p_2}=\frac{1}{p}+\frac{1}{q}$. Again, applying H\"{o}lder's inequality and Lemma \ref{Bernstein Lemma}, we get
\begin{align}\label{commutator estimate 2}
\|\dot{\Delta}_j(\dot{T}_{f}a)\|_{p_2}\lesssim \sum_{|j'-j|\le4}\|\dot{\Delta}_{j'}a\|_{q}\|\dot{S}_{j'-1}f\|_{p}\lesssim c_{j,r}2^{-js_2}\|a\|_{\dot{B}_{q,\infty}^{3/q}}\|f\|_{\dot{B}_{p_1,r}^{s_1}},
\end{align}
where in the second inequality we need the fact that $p\ge p_1$ and $\frac{3}{q}>s_2$. In the case $q\le p_2$, thanks to Lemma \ref{basic properties of Besov spaces} {\rm (\romannumeral1)}, the same result stays true whenever $\frac{3}{p_2}>s_2$.

For the third term in \eqref{split commutator}, we first assume $\frac{1}{p_1}+\frac{1}{q}\le1$ and redefine $p$ by $\frac{1}{p}=\frac{1}{p_1}+\frac{1}{q}$. Using Lemma \ref{Bernstein Lemma} and H\"{o}lder's inequality, and noticing that $p\le p_2$ and $\frac{3}{q}+s_1>0$, we obtain
\begin{align}\label{commutator estimate 3}
\|\dot{\Delta}_j(\dot{R}(a,f))\|_{p_2}\lesssim2^{3j(\frac1p-\frac{1}{p_2})}\sum_{j'\ge j-3}\|\dot{\Delta}_{j'}a\|_{q}\|\widetilde{\dot{\Delta}}_{j'}f\|_{p_1}\lesssim c_{j,r}2^{-js_2}\|a\|_{\dot{B}_{q,\infty}^{3/q}}\|f\|_{\dot{B}_{p_1,r}^{s_1}}.
\end{align}
If $\frac{1}{p_1}+\frac{1}{q}\ge1$, the result still holds provided that $\frac{3}{p'_1}+s_1>0$.

For the last term in \eqref{split commutator}, redefining $p$ by $\frac{1}{p_2}=\frac{1}{q}+\frac{1}{p}$ if $q\ge p_2$, we see that
\begin{align}\label{commutator estimate 4}
\|\dot{T}_{\dot{\Delta}_jf}^{'}a\|_{p_2}\lesssim \sum_{j'\ge j-2}\|\dot{\Delta}_{j'}a\|_{q}\|\dot{\Delta}_{j}f\|_{p}\lesssim c_{j,r}2^{-js_2}\|a\|_{\dot{B}_{q,\infty}^{3/q}}\|f\|_{\dot{B}_{p_1,r}^{s_1}},
\end{align}
where in the second inequality we need the condition $q<\infty$. If $p_2\ge q$, the same result holds under the assumption that $p_2<\infty$. 

Putting \eqref{commutator estimate 1}-\eqref{commutator estimate 4} together finishes the proof.
\end{proof}
\begin{remark}
The technical assumption $a\in L^\infty(\R^3)$ is needed in order for the product $af$ to be well-defined via paraproducts.
\end{remark}

With the above commutator estimates at our disposal, we are now able to prove the following elliptic regularity which will be used for iteration.

\begin{lemma}\label{Iteration}
Suppose that $r\in[1,\infty]$, $2\le p_2<p_1\le\infty$, $q\in[1,\infty)$, $(s_1,s_2)\in\R^2$, $s_1-\frac{3}{p_1}=s_2-\frac{3}{p_2}$, and
\begin{align*}
\frac{1}{p_2}\le\frac{1}{p_1}+\frac{1}{q},\  \frac{1}{p_2}<\frac{1}{p_1}+\frac{1}{3},\  \frac{3}{q\vee p_2}>s_2,\ \ \frac{3}{q\vee p'_{1}}>-s_1.
\end{align*}
Let $b$ satisfy $\nabla b\in\dot{B}_{q,\infty}^{3/q-1}(\R^3)$. If in addition, $f\in L^2(\R^3)\cap\dot{B}_{p_2,r}^{s_2}(\R^3)$ and $\nabla P\in \dot{B}_{p_1,r}^{s_1}(\R^3)$, then there exists a constant $C>0$ such that
\begin{align*}
\|\nabla P\|_{\dot{B}_{p_2,r}^{s_2}}\le C\|f\|_{\dot{B}_{p_2,r}^{s_2}}+C\|b\|_{\dot{B}_{q,\infty}^{3/q}}\|\nabla P\|_{\dot{B}_{p_1,r}^{s_1}}.
\end{align*}
\end{lemma}
\begin{proof}
Testing $v=\dot{\Delta}_j(|\dot{\Delta}_jP|^{p_2-2}\dot{\Delta}_jP)\in H^\infty(\R^3)$ in the equation 
\begin{align*}
\int_{\R^3}b\nabla P\cdot\nabla v\,dx=\int_{\R^3}f\cdot\nabla v\,dx, \end{align*}
we have
\begin{align*}
\int_{\R^3}b|\dot{\Delta}_j\nabla P|^2|\dot{\Delta}_jP|^{p_2-2}\,dx=\int_{\R^3}(\dot{\Delta}_jf-[\dot{\Delta}_j,b]\nabla P)\cdot\dot{\Delta}_j\nabla P|\dot{\Delta}_jP|^{p_2-2}\,dx.
\end{align*}
Applying {\cite[Lemma~A.5]{danchin cpde 2001}}, Lemma \ref{Bernstein Lemma} and H\"{o}lder's inequality, we get
\begin{align*}
\|\dot{\Delta}_j\nabla P\|_{p_2}\lesssim \|\dot{\Delta}_jf\|_{p_2}+\|[\dot{\Delta}_j,b]\nabla P\|_{p_2}.
\end{align*}
Multiplying both sides by $2^{js_2}$ and taking $l^r$ norm with respect to $j$, we obtain
\begin{align*}
\|\nabla P\|_{\dot{B}_{p_2,r}^{s_2}}\lesssim\|f\|_{\dot{B}_{p_2,r}^{s_2}}+\left\|\left(2^{js_2}\|[\dot{\Delta}_j,b]\nabla P\|_{p_2}\right)_j\right\|_{l_j^r}.
\end{align*}
The desired result then follows from Lemma \ref{commutator estimates}.
\end{proof}

We are now in a position to complete the proof of Theorem \ref{elliptic estimates in Besov}.
\begin{proof}[Proof of Theorem \ref{elliptic estimates in Besov}] We start with choosing $s_0$ and $p_0$. Since $(-\frac{3}{q})\vee(s-\frac{3}{p})<\mu-1$, we can choose an $s_0$ between both sides of this inequality. Define $p_0$ by $s_0-\frac{3}{p_0}=s-\frac{3}{p}$, then $p_0\in(p,\infty)$. By Lemma \ref{Initial iteration}, we have $\|\nabla P\|_{\dot{B}_{p_0,r}^{s_0}}\le C\|f\|_{\dot{B}_{p,r}^{s}}$. Next, we shall choose $(p_1,s_1)\in[p,p_0)\times(s_0,s]$ that satisfies the assumptions in Lemma \ref{Iteration}. It is not difficult to see that those assumptions can be reduced to $s_1-\frac{3}{p_1}=s-\frac{3}{p}$ and 
\begin{align*}
\frac{1}{p_1}\le\frac{1}{p_0}+\frac{1}{q},\  \frac{1}{p_1}<\frac{1}{p_0}+\frac{1}{3}.
\end{align*}
If $(p_1,s_1)=(p,s)$ satisfies the above assumptions, we are done by using Lemma \ref{Iteration}. Otherwise, we define $p_1$ by $\frac{1}{p_1}=\frac{1}{p_0}+\frac{1}{2(q\vee 3)}$, and get
\begin{align*}
\|\nabla P\|_{\dot{B}_{p_1,r}^{s_1}}\lesssim \|f\|_{\dot{B}_{p_1,r}^{s_1}}+\|b\|_{\dot{B}_{q,\infty}^{3/q}}\|\nabla P\|_{\dot{B}_{p_0,r}^{s_0}}\lesssim \|f\|_{\dot{B}_{p,r}^{s}}.
\end{align*}
The $(p_k,s_k)$ is defined by $\frac{1}{p_k}=\frac{1}{p_0}+\frac{k}{2(q\vee 3)}$ and $s_k-\frac{3}{p_k}=s-\frac{3}{p}$. So the iteration scheme will end in a finite number of steps. This completes the proof.
\end{proof}

\begin{remark}
If we know a priori that the norm $\|\nabla P\|_{\dot{B}_{p,r}^{s}}$ is finite, then the iteration process is not needed. We apply Lemma \ref{Iteration} only once to get
\begin{align*}
\|\nabla P\|_{\dot{B}_{p,r}^{s}}\le C\|f\|_{\dot{B}_{p,r}^{s}}+C\|b\|_{\dot{B}_{q,\infty}^{3/q}}\|\nabla P\|_{\dot{B}_{p_1,r}^{s_1}}
\end{align*}
with some $s_1\in(0,s)$. To complete the proof, we use Lemma \ref{basic properties of Besov spaces} (\romannumeral2) and Young's inequality. In this way, we can also remove the technical assumption that $q<\frac{3}{1-\mu}$. 
\end{remark}

Given Theorem \ref{elliptic estimates in Besov}, we are now able to identify $\mathcal{G}$ and its domain $D(\mathcal{G})$. In order to avoid some unpleasant technicalities, we will simply assume that $b-1\in\dot{B}_{q,1}^{3/q}(\R^3)$. Then by Lemma \ref{basic properties of Besov spaces} (\romannumeral5), $\rho-1=(1-b)/b$ satisfies the same assumption.
\begin{lemma}\label{identify G in besov}
Let $0<s<\frac{1}{2}+\mu$ and $1\le q<\frac{3}{s\vee(1-\mu)}$. Assume that $b-1\in\dot{B}_{q,1}^{3/q}(\R^3)$. Then $\mathcal{G}=\mathcal{G}_{s,1}$ coincides with the operator $\Tilde{\mathcal{G}}$ defined by
\begin{align*}
\Tilde{\mathcal{G}}=b\mathcal{P}_b\Delta:\mathcal{P}\dot{B}_{2,1}^{s}\cap\mathcal{P}\dot{B}_{2,1}^{s+2}\subset\mathcal{P}\dot{B}_{2,1}^{s}\rightarrow\mathcal{P}\dot{B}_{2,1}^{s}.
\end{align*}
\end{lemma}
\begin{proof}
First, we get from Theorem \ref{elliptic estimates in Besov} that $\nabla\mathcal{L}_b^{-1}\divg$ extends to a continuous operator on $\dot{B}_{2,1}^{s}$. In view of product laws in Besov spaces and Lemma \ref{basic properties of P and S} (\romannumeral2), $b\mathcal{P}_b$ is also continuous on $\dot{B}_{2,1}^{s}$, and the restriction of $b\mathcal{P}_b$ on $\mathcal{P}\dot{B}_{2,1}^{s}$ is invertible with a continuous inverse $\mathcal{P}\rho$. Based on this, it is not difficult to see that $\Tilde{\mathcal{G}}$ is a closed operator. Note that the space $\mathscr{C}$ (defined in Lemma \ref{core for G}) coincides with the inhomogeneous space $\mathcal{P}B_{2,1}^{s+2}$, so it is dense in $D(\Tilde{\mathcal{G}})$. This shows that $\Tilde{\mathcal{G}}$ is the closure of $\mathcal{S}:\mathscr{C}\subset\mathcal{P}\dot{B}_{2,1}^{s}\rightarrow\mathcal{P}\dot{B}_{2,1}^{s}$. So we have $\mathcal{G}=\Tilde{\mathcal{G}}$ as a consequence of Lemma \ref{core for G}.
\end{proof}

With a slight abuse of notation, we shall not distinguish between $\mathcal{S}$ and $\mathcal{G}$.

\subsection{Proof of Theorem \ref{maximal regularity for Stokes}}

We give a proof of Theorem \ref{maximal regularity for Stokes} in this subsection. To further simplify the exposition, we assume that $b-1\in\dot{B}_{2,1}^{3/2}(\R^3)$. First, let us go back to the maximal regularity for the abstract Cauchy problem
\begin{align}\label{ACP bPbf}
u'(t)-b\mathcal{P}_b\Delta u(t)=b\mathcal{P}_bf(t),\ \ u(0)=u_0.
\end{align}
As a consequence of Theorem \ref{maximal regularity for ACP rough b}, Theorem \ref{elliptic estimates in Besov} and Lemma \ref{identify G in besov}, we have:

\begin{corollary}\label{maximal regularity for ACP regular b}
Let $T\in(0,\infty]$. Assume that $b$ satisfies \eqref{initial density bounds} and $b-1\in\dot{B}_{2,1}^{3/2}(\R^3)$. Let $u_0\in\mathcal{P}\dot{B}_{2,1}^{1/2}(\R^3)$ and $f\in L^1((0,T);\dot{B}_{2,1}^{1/2}(\R^3))$. Then \eqref{ACP bPbf} has a unique strong solution $u\in C_b([0,T);\mathcal{P}\dot{B}_{2,1}^{1/2}(\R^3))$. Moreover, there exists a constant $C$ depending on $m$ and $\|b-1\|_{\dot{B}_{2,1}^{3/2}}$ such that
\begin{align*}
\|u\|_{L_T^\infty(\dot{B}_{2,1}^{1/2})}+\|u',\Delta u\|_{L_T^1(\dot{B}_{2,1}^{1/2})}\le C\|u_0\|_{\dot{B}_{2,1}^{1/2}}+C\|f\|_{L_T^1(\dot{B}_{2,1}^{1/2})}.
\end{align*}
\end{corollary}

Let us now give the proof of the well-posedness part of Theorem \ref{maximal regularity for Stokes}.

\begin{proof}[Proof of the well-posedness part of Theorem \ref{maximal regularity for Stokes}]
Note that $b=\rho^{-1}$ satisfies the same assumptions as $\rho$. By Corollary \ref{maximal regularity for ACP regular b}, we see that the following Cauchy problem
\begin{align*}
\partial_tv(t)-b\mathcal{P}_b\Delta v(t)=b\mathcal{P}_b(f-\rho\mathcal{Q}\partial_tR),\ v(0)=u_0
\end{align*}
has a unique strong solution $v\in C_b([0,T);\mathcal{P}\dot{B}_{2,1}^{1/2}(\R^3))$ satisfying
\begin{align}\label{rough maximal regularity estimate for v}
\|v\|_{L_T^\infty(\dot{B}_{2,1}^{1/2})}+\|\partial_tv,\Delta v\|_{L_T^1(\dot{B}_{2,1}^{1/2})}\lesssim \|u_0\|_{\dot{B}_{2,1}^{1/2}}+\|f,\partial_t R\|_{L_T^1(\dot{B}_{2,1}^{1/2})}.
\end{align}
Define $u=v+\mathcal{Q}R=v-\nabla(-\Delta)^{-1}g$ and $\nabla P=\mathcal{Q}(f-\rho\partial_t v-\rho\mathcal{Q}\partial_tR)+\nabla g$. One can readily check that $(u,\nabla P)$ is a strong solution to \eqref{Stokes system}.
\end{proof}

By \eqref{rough maximal regularity estimate for v} and the construction of $\nabla P$, we also have
\begin{align}\label{pressure estimate for P}
\|\nabla P\|_{L_T^1(\dot{B}_{2,1}^{1/2})}\lesssim \|u_0\|_{\dot{B}_{2,1}^{1/2}}+\|f,\partial_t R,\nabla g\|_{L_T^1(\dot{B}_{2,1}^{1/2})}.
\end{align}
But if we apply \eqref{rough maximal regularity estimate for v} to bound $u$ directly, we have to include the term $\|\mathcal{Q}R\|_{L_T^\infty(\dot{B}_{2,1}^{1/2})}$ on the right side of \eqref{maximal regularity estimates for Stokes}. This would cause serious trouble for us to prove local existence of large solutions to \eqref{Lagrangian formulation}. Thanks to \eqref{pressure estimate for P}, we can view $\nabla P$ in the first equation of \eqref{Stokes system} as a source term. So we choose to prove the maximal regularity for the solution $u$ to the parabolic system \eqref{write P as a source term}. Having had success in establishing maximal $L^1$ regularity for \eqref{ACP bPbf} based on Theorem \ref{characterization of besov spaces with positive regularity via stokes}, we are going to obtain maximal $L^1$ regularity for the parabolic Cauchy problem
\begin{align}\label{ACP b delta}
\partial_t u-b\Delta u=f,\ \ u(0)=u_0
\end{align}
by characterizations of Besov norms via the semigroup $e^{tb\Delta}$. To simplify the exposition, we only prove what is needed for the proof of \eqref{maximal regularity estimates for Stokes}.

\begin{lemma}\label{maximal parabolic regularity}
Suppose that $b$ satisfies \eqref{initial density bounds}, $b$ and $b^{-1}$ are multipliers of $\dot{B}_{2,1}^{1/2}(\R^3)$ (i.e., multiplications by $b$ and $b^{-1}$ are continuous on $\dot{B}_{2,1}^{1/2}(\R^3)$). Let $u_0\in \dot{B}_{2,1}^{1/2}(\R^3)$ and $f\in L^1((0,T);\dot{B}_{2,1}^{1/2}(\R^3))$. Then \eqref{ACP b delta} has a unique strong solution $u$ satisfying
\begin{align*}
\|u\|_{L_T^\infty(\dot{B}_{2,1}^{1/2})}+\|\partial_tu,\Delta u\|_{L_T^1(\dot{B}_{2,1}^{1/2})}\lesssim \|u_0\|_{\dot{B}_{2,1}^{1/2}}+\|f\|_{L_T^1(\dot{B}_{2,1}^{1/2})}.
\end{align*}
\end{lemma}
\begin{proof}
Since the proof is analogous to that of Theorem \ref{maximal regularity for ACP rough b}, we only outline the key steps.

First, $b\Delta:H^2(\R^3)\subset L^2(\R^3)\rightarrow L^2(\R^3)$ generates a bounded analytic semigroup $e^{tb\Delta}$ whose kernel has a Gaussian upper bound (see \cite{mcintosh 2000,duong DIE 1999}). So we have $\lim_{t\rightarrow\infty}\|e^{tb\Delta}f\|=0$ for every $f\in L^2(\R^3)$. Based on this, we can mimic the proof of Theorem \ref{characterization of besov spaces with positive regularity via stokes} to get the equivalence of norms:
\begin{align*}
\|u_0\|_{\dot{B}_{2,1}^{1/2}}\simeq\left\|t^{-1/4}\|tb\Delta e^{tb\Delta}u_0\|\right\|_{L^1(\R_+,\frac{dt}{t})}, \ \ \forall u_0\in H^2(\R^3).
\end{align*}
This would imply that \begin{align*}
\sup_{t>0}\|e^{tb\Delta} u_0\|_{\dot{B}_{2,1}^{1/2}}+\sup_{t>0}\|tb\Delta e^{tb\Delta} u_0\|_{\dot{B}_{2,1}^{1/2}}+\left\|\|tb\Delta e^{tb\Delta}u_0\|_{\dot{B}_{2,1}^{1/2}}\right\|_{L^1(\R_+,\frac{dt}{t})}\le C\|u_0\|_{\dot{B}_{2,1}^{1/2}},
\end{align*}
for all $u_0\in H^2(\R^3)$. So $e^{tb\Delta}|_{H^2(\R^3)}$ extends to a bounded analytic semigroup $\mathcal{T}(t)$ on $\dot{B}_{2,1}^{1/2}(\R^3)$. Denote by $\mathcal{G}$ the generator of $\mathcal{T}(t)$. Then $\mathscr{C}\vcentcolon=\{u_0\in H^2(\R^3)|b\Delta u_0\in\dot{B}_{2,1}^{1/2}(\R^3)\}$ is a core for $\mathcal{G}$, and $\mathcal{G}|_{\mathscr{C}}=b\Delta|_{\mathscr{C}}$. So far, all statements hold if $b$ merely satisfies \eqref{initial density bounds}.

Next, we assume that both $b$ and $b^{-1}$ are multipliers of $\dot{B}_{2,1}^{1/2}(\R^3)$. Then $\mathcal{G}$ coincides with the operator
\begin{align*}
b\Delta:\dot{B}_{2,1}^{1/2}(\R^3)\cap\dot{B}_{2,1}^{5/2}(\R^3)\subset\dot{B}_{2,1}^{1/2}(\R^3)\rightarrow\dot{B}_{2,1}^{1/2}(\R^3).
\end{align*}

Now that all preparation work is done, we mimic the proof of Theorem \ref{maximal regularity for ACP rough b} to finish the proof of the present lemma. The details are left to the reader.
\end{proof}

We conclude this section by completing the proof of Theorem \ref{maximal regularity for Stokes}
\begin{proof}[Proof of Theorem \ref{maximal regularity for Stokes}] It remains to show \eqref{maximal regularity estimates for Stokes}. Applying Lemma \ref{maximal parabolic regularity} to \eqref{write P as a source term}, we have
\begin{align*}
\|u\|_{L_T^\infty(\dot{B}_{2,1}^{1/2})}+\|\partial_tu,\Delta u\|_{L_T^1(\dot{B}_{2,1}^{1/2})}\lesssim \|u_0\|_{\dot{B}_{2,1}^{1/2}}+\|f\|_{L_T^1(\dot{B}_{2,1}^{1/2})}+\|\nabla P\|_{L_T^1(\dot{B}_{2,1}^{1/2})}.
\end{align*}
This together with \eqref{pressure estimate for P} gives \eqref{maximal regularity estimates for Stokes}. Thus, the proof of Theorem \ref{maximal regularity for Stokes} is completed.
\end{proof}

\bigskip

\section{Local and global well-posedness}\label{proof of main theorems}

\bigskip

In this section, we prove the well-posedness of \eqref{Lagrangian formulation}, then the well-posedness of \eqref{INS} will follow. Let $E(T)$ denote the space of all pairs $(\lagr{u},\nabla\lagr{P})$ satisfying
\begin{align*}
\lagr{u}\in C([0,T];\dot{B}_{2,1}^{1/2}(\R^3)),\ 
(\partial_t\lagr{u}, \Delta\lagr{u}, \nabla \lagr{P})\in \left(L^{1}((0,T);\dot{B}_{2,1}^{1/2}(\R^3))\right)^3,  
\end{align*}
endowed with the norm
\begin{align*}
\|(\lagr{u},\nabla\lagr{P})\|_{E(T)}=\|\lagr{u}\|_{L_T^\infty(\dot{B}_{2,1}^{1/2})}+\|\Delta\lagr{u},\partial_t\lagr{u},\nabla \lagr{P}\|_{L_T^1(\dot{B}_{2,1}^{1/2})}.
\end{align*}

Given Theorem \ref{maximal regularity for Stokes}, we can prove the well-posedness of \eqref{Lagrangian formulation} in $E(T)$ by using the contraction mapping theorem.

\begin{theorem}\label{local wellposedness for lagrangian formulation}
Assume that the initial density $\rho_0$ satisfies \eqref{initial density bounds} and $\rho_0-1\in\dot{B}_{2,1}^{3/2}(\R^3)$, and the initial velocity $u_0\in\mathcal{P}\dot{B}_{2,1}^{1/2}(\R^3)$. Then there exists some $T>0$, such that the system \eqref{Lagrangian formulation} has a unique strong solution $(\lagr{u},\nabla \lagr{P})\in E(T)$.
\end{theorem}
\begin{proof}
We shall construct a contraction mapping on $E(T)$ by solving the linearized system \eqref{linearized Lagrangian formulation}. Let us denote the inhomogeneous terms in \eqref{linearized Lagrangian formulation} by
\begin{eqnarray*}
\left\{\begin{aligned}
&f(\lagr{v},\nabla\lagr{Q})=\divg((\mathscr{A}_{\lagr{v}}\mathscr{A}_{\lagr{v}}^{T}-I)\nabla\lagr{v})+(I-\mathscr{A}_{\lagr{v}}^{T})\nabla\lagr{Q},\\
&R(\lagr{v})=(I-\mathscr{A}_{\lagr{v}})\lagr{v},\\
&g(\lagr{v})=\mathrm{Tr}((I-\mathscr{A}_{\lagr{v}})D\lagr{v}),
\end{aligned}\right.
\end{eqnarray*}
where $(\lagr{v},\nabla\lagr{Q})\in E(T)$ and $I$ is a $3\times3$ identity matrix. However, since no smallness is assumed on the initial data, one has to perform the contraction mapping theorem around a neighborhood of some reference, which here is chosen as the solution $(\lagr{u}_L,\nabla\lagr{P}_L)$ to the homogeneous linear system
\begin{eqnarray}\label{homogeneous stokes system}
\left\{\begin{aligned}
&\rho_0\partial_t \lagr{u}_L-\Delta\lagr{u}_L+\nabla\lagr{P}_L=0,\\
&\divg\lagr{u}_L=0,\\
&\lagr{u}_L(0,\cdot)=u_0.
\end{aligned}\right.
\end{eqnarray}
In view of Theorem \ref{maximal regularity for Stokes}, we immediately have
\begin{align}\label{maximal regularity for free solutions}
\|(\lagr{u}_L,\nabla\lagr{P}_L)\|_{E(T)}
\le C\|u_0\|_{\dot{B}_{2,1}^{1/2}}.
\end{align}
Hence, we see that
\begin{align*}
M(t)\vcentcolon=\|\Delta\lagr{u}_L,\partial_t\lagr{u}_L,\nabla\lagr{P}_L\|_{L_{t}^{1}(\dot{B}_{2,1}^{1/2})}\rightarrow0
\end{align*}
as $t$ tends to $0$. We will solve \eqref{linearized Lagrangian formulation} in a closed ball in $E(T)$ centered at $(\lagr{u}_L,\nabla\lagr{P}_L)$ with radius $r$, that is,
\begin{align*}
B_r(\lagr{u}_L,\nabla\lagr{P}_L)=\{(\lagr{u},\nabla\lagr{P})\in E(T):\|(\bar{\lagr{u}},\nabla\bar{\lagr{P}})\|_{E(T)}\le r\},
\end{align*}
where $(\bar{\lagr{u}},\bar{\lagr{P}})=(\lagr{u}-\lagr{u}_L,\lagr{P}-\lagr{P}_L)$. 
The numbers $r$ and $T$ will be chosen suitably small later.

Let us estimate the inhomogeneous terms first. For any $(\lagr{v},\nabla\lagr{Q})\in B_r(\lagr{u}_L,\nabla\lagr{P}_L)$, we denote $(\bar{\lagr{v}},\bar{\lagr{Q}})=(\lagr{v}-\lagr{u}_L,\lagr{Q}-\lagr{P}_L)$. Obviously, we have
\begin{align*}
\|\nabla\lagr{v}\|_{L_T^1(\dot{B}_{2,1}^{3/2})}\le\|\nabla\bar{\lagr{v}}\|_{L_T^1(\dot{B}_{2,1}^{3/2})}+\|\nabla\lagr{u}_L\|_{L_T^1(\dot{B}_{2,1}^{3/2})}\le r+M(T).
\end{align*}
So Lemma \ref{inverse flow estimates} and Lemma \ref{flow estimates} are effective if we require
\begin{align*}
r+M(T)\le c_0.
\end{align*}
Then applying \eqref{flow estimate for A-I} and product laws in Besov spaces, we see that
\begin{align}\label{estimate for inhomogeneous terms f and g}
\|f(\lagr{v},\nabla\lagr{Q}),\nabla g(\lagr{v})\|_{L_T^1(\dot{B}_{2,1}^{1/2})}\lesssim\|\nabla\lagr{v}\|_{L_T^1(\dot{B}_{2,1}^{3/2})}\|\Delta\lagr{v},\nabla\lagr{Q}\|_{L_T^1(\dot{B}_{2,1}^{1/2})}.
\end{align}
While applying \eqref{flow estimate for A-I} and \eqref{flow estimate for partial t A}, we get
\begin{align}
\|\partial_tR(\lagr{v})\|_{L_T^1(\dot{B}_{2,1}^{1/2})}\lesssim&\int_0^T\|\partial_t\mathscr{A}_{\lagr{v}}\|_{\dot{B}_{2,1}^{1/2}}\|\lagr{v}\|_{\dot{B}_{2,1}^{3/2}}+\|I-\mathscr{A}_{\lagr{v}}\|_{\dot{B}_{2,1}^{3/2}}\|\partial_t\lagr{v}\|_{\dot{B}_{2,1}^{1/2}}\,dt\nonumber\\
\lesssim&\|\lagr{v}\|_{L_T^2(\dot{B}_{2,1}^{3/2})}^2+\|\nabla\lagr{v}\|_{L_T^1(\dot{B}_{2,1}^{3/2})}\|\partial_t\lagr{v}\|_{L_T^1(\dot{B}_{2,1}^{1/2})}\label{estimate for inhomogeneous term R}.
\end{align}
So, by Theorem \ref{maximal regularity for Stokes}, the system \eqref{linearized Lagrangian formulation} has a unique solution $(\lagr{u},\nabla\lagr{P})\in E(T)$. Subtracting \eqref{homogeneous stokes system} from \eqref{linearized Lagrangian formulation}, and then applying again Theorem \ref{maximal regularity for Stokes} to the resulting system for $(\bar{\lagr{u}},\nabla\bar{\lagr{P}})$, we obtain
\begin{align*}
\|(\bar{\lagr{u}},\nabla\bar{\lagr{P}})\|_{E(T)}\lesssim\|f(\lagr{v},\nabla\lagr{Q}),\partial_tR(\lagr{v}),\nabla g(\lagr{v})\|_{L_T^1(\dot{B}_{2,1}^{1/2})}.
\end{align*}
Plugging \eqref{estimate for inhomogeneous terms f and g} and \eqref{estimate for inhomogeneous term R} in the above estimate, and using Lemma \ref{basic properties of Besov spaces} (\romannumeral2) and \eqref{maximal regularity for free solutions}, we arrive at
\begin{align*}
\|(\bar{\lagr{u}},\nabla\bar{\lagr{P}})\|_{E(T)}\lesssim&\|\lagr{v}\|_{L_T^2(\dot{B}_{2,1}^{3/2})}^2+ \|\nabla\lagr{v}\|_{L_T^1(\dot{B}_{2,1}^{3/2})}\|\Delta\lagr{v},\partial_t\lagr{v},\nabla\lagr{Q}\|_{L_T^1(\dot{B}_{2,1}^{1/2})}\\
\le& C_0\|u_0\|_{\dot{B}_{2,1}^{1/2}}M(T)+C_0(r+M(T))^2,
\end{align*}
where $C_0$ depends on $m$ and $\|\rho_0-1\|_{\dot{B}_{2,1}^{3/2}}$. Now choosing $T$ and $r$ small enough so that
\begin{align}\label{smallness condition on r and T}
r\le\frac{1}{8C_0}\wedge\frac{c_0}{2}\ \ \mathrm{and}\ \ M(T)\le r\wedge\frac{r}{2C_0\|u_0\|_{\dot{B}_{2,1}^{1/2}}}, 
\end{align}
we have $\|(\bar{\lagr{u}},\nabla\bar{\lagr{P}})\|_{E(T)}\le r$. So the solution mapping $\mathcal{N}$ that assigns $(\lagr{v},\nabla\lagr{Q})$ to $(\lagr{u},\nabla\lagr{P})$  is a self-map on $B_r(\lagr{u}_L,\nabla\lagr{P}_L)$.

It remains to show that $\mathcal{N}$ is contractive. Let $(\lagr{v}_i,\nabla\lagr{Q}_i)\in B_r(\lagr{u}_L,\nabla\lagr{P}_L)$ and $(\lagr{u}_i,\nabla\lagr{P}_i)=\mathcal{N}(\lagr{v}_i,\nabla\lagr{Q}_i)$, $i=1,2$. In what follows, for two quantities $q_1$ and $q_2$, $\delta q$ always denotes their difference $q_1-q_2$. Then the system for $(\delta\lagr{u},\nabla\delta \lagr{P})$ reads
\begin{eqnarray}\label{Stokes system for difference of two solutions}
\left\{\begin{aligned}
&\rho_0\partial_t\delta\lagr{u}-\Delta\delta\lagr{u}+\nabla\delta\lagr{P}=\delta f,\\
&\divg\delta\lagr{u}=\divg\delta R=\delta g,\\
&\delta\lagr{u}|_{t=0}=0,
\end{aligned}\right.
\end{eqnarray}
where $f_i=f(\lagr{v}_i,\nabla\lagr{Q}_i)$, $g_i=g(\lagr{v}_i)$, and $R_i=R(\lagr{v}_i)$ with $\mathscr{A}_{i}=\mathscr{A}_{\lagr{v}_i}$. 

We write  $\delta f=(\delta f)_1+(\delta f)_2$, where
\begin{align*}
(\delta f)_1=&\divg((\mathscr{A}_{\lagr{v}_1}\mathscr{A}_{\lagr{v}_1}^T-I)\nabla \delta\lagr{v})+(I-\mathscr{A}_{\lagr{v}_1}^T)\nabla\delta\lagr{Q},\\
(\delta f)_2=&-(\delta \mathscr{A})^T\nabla\lagr{Q}_2+\divg\big[(\mathscr{A}_{\lagr{v}_1}\mathscr{A}_{\lagr{v}_1}^T-\mathscr{A}_{\lagr{v}_2}\mathscr{A}_{\lagr{v}_2}^T)\nabla\lagr{v}_2\big].
\end{align*}
Along the lines of deriving \eqref{estimate for inhomogeneous terms f and g}, we have
\begin{align*}
\|(\delta f)_1\|_{L_T^1(\dot{B}_{2,1}^{1/2})}\lesssim\|\nabla\lagr{v}_1\|_{L_T^1(\dot{B}_{2,1}^{3/2})}\|\Delta\delta\lagr{v},\nabla\delta\lagr{Q}\|_{L_T^1(\dot{B}_{2,1}^{1/2})}.
\end{align*}
Applying \eqref{flow estimate for A-I}, \eqref{flow estimate for difference of A} and product laws in Besov spaces, we obtain
\begin{align*}
\|(\delta f)_2\|_{L_T^1(\dot{B}_{2,1}^{1/2})}\lesssim\|\nabla\delta\lagr{v}\|_{L_T^1(\dot{B}_{2,1}^{3/2})}\|\Delta\lagr{v}_2,\nabla\lagr{Q}_2\|_{L_T^1(\dot{B}_{2,1}^{1/2})}.
\end{align*}
Summing up the estimates, we have
\begin{align}\label{estimate of difference of inhomogeneous term f}
\|\delta f\|_{L_T^1(\dot{B}_{2,1}^{1/2})}\lesssim(M(T)+r)\|(\delta\lagr{v},\nabla\delta\lagr{Q})\|_{E(T)}.  
\end{align}
Note that $\partial_t\delta R=-\partial_t\mathscr{A}_{\lagr{v}_1}\delta\lagr{v}+(I-\mathscr{A}_{\lagr{v}_1})\partial_t\delta\lagr{v}-\partial_t(\delta\mathscr{A})\lagr{v}_2-\delta\mathscr{A}\partial_t\lagr{v}_2$. Again, applying \eqref{flow estimate for A-I}, \eqref{flow estimate for partial t A}, \eqref{flow estimate for difference of A} and \eqref{flow estimate for partial t difference of A L2} gives
\begin{align*}
\|\partial_t\mathscr{A}_{\lagr{v}_1}\delta\lagr{v}\|_{L_T^1(\dot{B}_{2,1}^{1/2})}\lesssim&\|\nabla\lagr{v}_1\|_{L_T^1(\dot{B}_{2,1}^{3/2})}\|\delta\lagr{v}\|_{L_T^\infty(\dot{B}_{2,1}^{1/2})},\\
\|(I-\mathscr{A}_{\lagr{v}_1})\partial_t\delta\lagr{v}\|_{L_T^1(\dot{B}_{2,1}^{1/2})}\lesssim&\|\nabla\lagr{v}_1\|_{L_T^1(\dot{B}_{2,1}^{3/2})}\|\partial_t\delta\lagr{v}\|_{L_T^1(\dot{B}_{2,1}^{1/2})},\\
\|\partial_t(\delta\mathscr{A})\lagr{v}_2\|_{L_T^1(\dot{B}_{2,1}^{1/2})}\lesssim&\|\delta\lagr{v}\|_{L_T^2(\dot{B}_{p,1}^{3/2})}\|\lagr{v}_2\|_{L_T^2(\dot{B}_{2,1}^{3/2})},\\
\|\delta\mathscr{A}\partial_t\lagr{v}_2\|_{L_T^1(\dot{B}_{2,1}^{1/2})}\lesssim&\|\nabla\delta\lagr{v}\|_{L_T^1(\dot{B}_{2,1}^{3/2})}\|\partial_t\lagr{v}_2\|_{L_T^1(\dot{B}_{2,1}^{1/2})}.
\end{align*}
Putting things together, and using \eqref{maximal regularity for free solutions} and interpolation inequality in Besov spaces, we arrive at
\begin{align}\label{estimate of difference of inhomogeneous term R}
\|\partial_t\delta R\|_{L_T^1(\dot{B}_{2,1}^{1/2})}\lesssim(\|u_0\|_{\dot{B}_{2,1}^{1/2}}M^{1/2}(T)+M(T)+r)\|(\delta\lagr{v},\nabla\delta\lagr{Q})\|_{E(T)}. 
\end{align}
For the estimate of $\delta g$, we write $\delta g=\mathrm{Tr}((I-\mathscr{A}_{\lagr{v}_1})D\delta\lagr{v})-\mathrm{Tr}(\delta\mathscr{A}D\lagr{v}_2)$. We have
\begin{align}\label{estimate of difference of inhomogeneous term g}
\|\delta g\|_{L_T^1(\dot{B}_{2,1}^{3/2})}\lesssim&\|\nabla\lagr{v}_1\|_{L_T^1(\dot{B}_{2,1}^{3/2})}\|\nabla\delta\lagr{v}\|_{L_T^1(\dot{B}_{2,1}^{3/2})}+\|\nabla\lagr{v}_2\|_{L_T^1(\dot{B}_{2,1}^{3/2})}\|\nabla\delta\lagr{v}\|_{L_T^1(\dot{B}_{2,1}^{3/2})}\nonumber\\
\lesssim&(M(T)+r)\|(\delta\lagr{v},\nabla\delta\lagr{Q})\|_{E(T)}.
\end{align}

Now summing up \eqref{estimate of difference of inhomogeneous term f}--\eqref{estimate of difference of inhomogeneous term g} and applying Theorem \ref{maximal regularity for Stokes} to \eqref{Stokes system for difference of two solutions}, we obtain
\begin{align*}
\|(\delta\lagr{u},\nabla\delta\lagr{P})\|_{E(T)}\lesssim&\|\delta f,\partial_t\delta R,\nabla\delta g\|_{L_T^1(\dot{B}_{2,1}^{1/2})}\\
\le& C_1(\|u_0\|_{\dot{B}_{2,1}^{1/2}}M^{1/2}(T)+M(T)+r)\|(\delta\lagr{v},\nabla\delta\lagr{Q})\|_{E(T)}
\end{align*}
with $C_1$ depending on $m$ and $\|\rho_0-1\|_{\dot{B}_{2,1}^{3/2}}$. Taking \eqref{smallness condition on r and T} into consideration, we choose $r$ and $T$ so small that
\begin{align*}
r\le\frac{c_0}{2}\wedge\frac{1}{8C_0}\wedge\frac{1}{8C_1}\ \ \mathrm{and}\ \ M(T)\le r\wedge\frac{r}{2C_0\|u_0\|_{\dot{B}_{2,1}^{1/2}}}\wedge\left(4C_1\|u_0\|_{\dot{B}_{2,1}^{1/2}}\right)^{-2}.
\end{align*}
Then $\mathcal{N}$ is a contraction mapping on $B_r(\lagr{u}_L,\nabla\lagr{P}_L)$. So it admits a unique fixed point $(\lagr{u},\nabla\lagr{P})$ in $B_r(\lagr{u}_L,\nabla\lagr{P}_L)$, which is a solution to \eqref{Lagrangian formulation} in $E(T)$. The proof of uniqueness in $E(T)$ is similar to the stability estimates. So the proof of the theorem is completed.
\end{proof}

\begin{theorem}\label{global wellposedness for lagrangian formulation}
Under the assumptions in Theorem \ref{local wellposedness for lagrangian formulation}, there exists a constant $\varepsilon_0$ depending on $m$ and $\|\rho_0-1\|_{\dot{B}_{2,1}^{3/2}}$ such that if
\begin{align*}
\|u_0\|_{\dot{B}_{2,1}^{1/2}}\le\varepsilon_0,
\end{align*}
then the local solution $(\lagr{u},\nabla\lagr{P})$ exists globally in time and verifies
\begin{align*}
\|(\lagr{u},\nabla\lagr{P})\|_{E(\infty)}\vcentcolon=\|\lagr{u}\|_{L^\infty(\R_+;\dot{B}_{2,1}^{1/2})}+\|\Delta\lagr{u},\partial_t\lagr{u},\nabla \lagr{P}\|_{L^1(\R_+;\dot{B}_{2,1}^{1/2})}\le C\|u_0\|_{\dot{B}_{2,1}^{1/2}}.
\end{align*}
\end{theorem}
\begin{proof}
The proof is almost the same as that of Theorem \ref{local wellposedness for lagrangian formulation}. Let us just mention a few modifications. First, we should replace $E(T)$ by $E(\infty)$ that consists of all pairs $(\lagr{u},\nabla\lagr{P})$ satisfying
\begin{align*}
\lagr{u}\in C_b([0,\infty);\dot{B}_{2,1}^{1/2}(\R^3)),\ 
(\partial_t\lagr{u}, \Delta\lagr{u}, \nabla \lagr{P})\in \left(L^{1}(\R_+;\dot{B}_{2,1}^{1/2}(\R^3))\right)^3. 
\end{align*}
Second, we replace the reference solution $(\lagr{u}_L,\nabla\lagr{P}_L)$ by $(0,0)$, and choose $r$ as a small number depending on $m$ and $\|\rho_0-1\|_{\dot{B}_{2,1}^{3/2}}$ and $\|u_0\|_{\dot{B}_{2,1}^{1/2}}\lesssim r$. The details are left to the reader.
\end{proof}

For completeness, let us now give the proof of Theorem \ref{global wellposedness} and Theorem \ref{local and global well posedness for Lagragian formulation}.

\begin{proof}[Proof of Theorem \ref{local and global well posedness for Lagragian formulation}] Theorem \ref{local wellposedness for lagrangian formulation} and Theorem \ref{global wellposedness for lagrangian formulation} constitute a proof of Theorem \ref{local and global well posedness for Lagragian formulation}. 
\end{proof}

\begin{proof}[Proof of Theorem \ref{global wellposedness}] All details of the proof are provided in Appendix \ref{appendix1} except the uniqueness. For two solutions $(\rho_i,u_i,\nabla P_i)$, $i=1,2$, to \eqref{INS} with the same initial value, the corresponding solutions $(\lagr{u}_i,\nabla\lagr{P}_i)$ to \eqref{Lagrangian formulation} are identical in a short time interval, thus, so are $(\rho_i,u_i,\nabla P_i)$. Uniqueness on the interval where $(\rho_i,u_i,\nabla P_i)$ are defined can be proved in a standard way.
\end{proof}

\bigskip

\section{Long-time asymptotics}\label{proof of long time behavior}

\bigskip

This section is devoted to the proof of Theorem \ref{long time asymptotics}. Besides the maximal regularity estimate \eqref{maximal regularity INS}, the proof also relies on a recent result in \cite{zhang ping adv 2020}. 

\begin{lemma}\label{combine with zhang's result}
Assume that $\rho_0$ satisfies \eqref{initial density bounds} and $\rho_0-1\in\dot{B}_{2,1}^{3/2}(\R^3)$, and $u_0\in\mathcal{P}\dot{B}_{2,1}^{1/2}(\R^3)$. There exists a constant $\varepsilon_1$ depending on $m$ and $\|\rho_0-1\|_{\dot{B}_{2,1}^{3/2}}$ such that if $u_0$ satisfies
\begin{align*}
\|u_0\|_{\dot{B}_{2,1}^{1/2}}\le\varepsilon_1,
\end{align*}
then \eqref{INS} has a global solution $(\rho,u,\nabla P)$ that verifies \eqref{maximal regularity INS}, \eqref{global in time estimate for the density} and 
\begin{align}\label{zhang's estimate}
\|u\|_{L^\infty(\R_+;\dot{B}_{2,1}^{1/2})}+\|\sqrt{t}(\partial_tu+u\cdot\nabla u)\|_{L^2(\R_+;\dot{B}_{2,1}^{1/2})}\le C_2\|u_{0}\|_{\dot{B}_{2,1}^{1/2}},
\end{align}
where $C_2$ is a constant depending only on $m$.
\end{lemma}
\begin{proof}
Let $\varepsilon_1$ be so small that the Theorem \ref{global wellposedness} in the present paper and the Theorem 1.2 in \cite{zhang ping adv 2020} hold. We start with mollifying the data by defining
\begin{align*}
\rho_{0,N}=1+\sum_{|j|\le N}\Dot{\Delta}_j(\rho_0-1)\ \ \mathrm{and}\ \ u_{0,N}=\sum_{|j|\le N}\Dot{\Delta}_ju_0.
\end{align*}
As in \cite{zhang ping adv 2020}, the above data generates a global solution $(\rho_N,u_N,\nabla P_N)$ to \eqref{INS} that satisfies the estimates in {\cite[Theorem~1.2]{zhang ping adv 2020}}. On the other hand, in view of Theorem \ref{global wellposedness}, $(\rho_N,u_N,\nabla P_N)$ also satisfies \eqref{maximal regularity INS} and \eqref{global in time estimate for the density}. The uniform estimates allow us to pass to a limit to obtain a global strong solution $(\rho,u,\nabla P)$ to \eqref{INS}, which is unique due to Theorem \ref{global wellposedness}. Finally, we use the estimates in {\cite[Theorem~1.2]{zhang ping adv 2020}} to get
\begin{align*}
&\|\sqrt{t}(\partial_tu+u\cdot\nabla u)\|_{L^2(\R_+;\dot{B}_{2,1}^{1/2})}\\
\lesssim&\|\sqrt{t}\partial_tu\|_{L^2(\R_+;\dot{B}_{2,1}^{1/2})}+\|\sqrt{t}u\|_{L^\infty(\R_+;\dot{B}_{2,1}^{3/2})}\|u\|_{L^2(\R_+;\dot{B}_{2,1}^{3/2})}\lesssim\|u_{0}\|_{\dot{B}_{2,1}^{1/2}}.
\end{align*}
This completes the proof of the lemma.
\end{proof}

We are now in a position to give the proof of Theorem \ref{long time asymptotics}. Our proof is motivated by \cite{gallagher 2002,gallagher 2003} concerning asymptotics and stability for global solutions to the classical Navier-Stokes equations.

\begin{proof}[Proof of Theorem \ref{long time asymptotics}]
Fix any $\varepsilon<\varepsilon_1$. We first split the initial velocity into two parts
$u_0=u_{0,h}+u_{0,l}$, where $u_{0,h}=\sum_{j\ge-N}\dot{\Delta}_ju_0$ is the high frequency part that belongs to the inhomogeneous Besov space $ B^{1/2}_{2,1}(\R^3)$, while $u_{0,l}$ satisfies that 
\begin{align*}
\|u_{0,l}\|_{\dot{B}^{1/2}_{2,1}}\le \varepsilon.
\end{align*}
By Lemma \ref{combine with zhang's result}, $(\rho_0,u_{0,l})$ generates a global solution $(\rho_l,u_l,\nabla P_l)$ to \eqref{INS} that satisfies
\begin{align}\label{estimate for low frequency velocity}
\|u_l\|_{L^\infty(\R_+;\dot{B}_{2,1}^{1/2})}+\|\nabla u_l\|_{L^1(\R_+;\dot{B}_{2,1}^{3/2})}+\|\sqrt{t}(\partial_tu+u\cdot\nabla u)\|_{L^2(\R_+;\dot{B}_{2,1}^{1/2})}\lesssim\|u_{0,l}\|_{\dot{B}_{2,1}^{1/2}},
\end{align}
and
\begin{align}\label{estimate for low frequency density}
\|\rho_l-1\|_{L^\infty(\R_+;\dot{B}_{2,1}^{1/2})}\lesssim\|\rho_0-1\|_{\dot{B}_{2,1}^{3/2}}. 
\end{align}

Let $(\rho_h,u_h,P_h)=(\rho-\rho_l,u-u_l,P-P_l)$. Then $(\rho_h,u_h,P_h)$ satisfies the system
\begin{eqnarray}\label{system for high frequency velocity}
\left\{\begin{aligned}
&\partial_t\rho_h+u\cdot\nabla\rho_h+u_h\cdot\nabla\rho_l=0,\\
&\rho(\partial_tu_h+u\cdot\nabla u_h)-\Delta u_h+\nabla P_h=-\rho u_h\cdot\nabla u_l-\rho_h(\partial_tu_l+u_l\cdot\nabla u_l),\\
&\divg u_h=0,\\
&(\rho_h,u_h)|_{t=0}=(0,u_{0,h}).
\end{aligned}\right.
\end{eqnarray}
We shall use the energy method to derive an $L_t^4(\dot{B}_{2,1}^{1/2})$ estimate for $u_h$. Taking the $L^2$ inner product of the second equation in \eqref{system for high frequency velocity} with $u_h$, and using H\"{o}lder's inequality and Sobolev inequality, we have
\begin{align*}
\frac{1}{2}\frac{d}{dt}\|\sqrt{\rho}u_h\|_2^2+\|\nabla u_h\|_2^2
=&-\int\rho(u_h\cdot\nabla u_l)\cdot u_h\,dx-\int\rho_h(\partial_tu_l+u_l\cdot\nabla u_l)\cdot u_h\,dx\\
\lesssim&\|\nabla u_l\|_\infty\|\sqrt{\rho}u_h\|_2^2+\|\rho_h\|_2\|\partial_tu_l+u_l\cdot\nabla u_l\|_3\|\nabla u_h\|_2.
\end{align*}
To bound $\|\rho_h\|_2$, we get from the transport equation in \eqref{system for high frequency velocity} that
\begin{align*}
\|\rho_h(t)\|_{2}\le\int_0^t\|u_h\cdot\nabla\rho_l\|_{2}\,d\tau\lesssim \int_0^t\|\nabla u_h\|_2\|\nabla\rho_l\|_{3}\,d\tau\lesssim \sqrt{t}\|\nabla u_h\|_{L_t^2(L^2)}\|\nabla\rho_l\|_{L_t^\infty(\dot{B}_{2,1}^{1/2})}.
\end{align*}
Putting things together and using \eqref{estimate for low frequency density}, we have
\begin{align*}
\frac{d}{dt}\|\sqrt{\rho}u_h\|_2^2+\|\nabla u_h\|_2^2
\lesssim\|\nabla u_l\|_\infty\|\sqrt{\rho}u_h\|_2^2+\|\nabla u_h\|_{L_t^2(L^2)}\|\sqrt{t}(\partial_tu_l+u_l\cdot\nabla u_l)\|_{3}\|\nabla u_h\|_2.
\end{align*}
Integrating both sides of the above inequality over the time interval $[0,t]$, then using \eqref{estimate for low frequency velocity}, we have
\begin{align*}
\|\sqrt{\rho}u_h\|_2^2(t)+\|\nabla u_h\|_{L_t^2(L^2)}^2
\lesssim\|\sqrt{\rho_0}u_{0,h}\|_2^2+\int_0^t\|\nabla u_l\|_\infty\|\sqrt{\rho}u_h\|_2^2\,d\tau+\varepsilon\|\nabla u_h\|_{L_t^2(L^2)}^2.
\end{align*}
So, if $\varepsilon$ is small enough, we arrive at
\begin{align*}
\|\sqrt{\rho}u_h\|_2^2(t)+\|\nabla u_h\|_{L_t^2(L^2)}^2
\lesssim\|\sqrt{\rho_0}u_{0,h}\|_2^2+\int_0^t\|\nabla u_l\|_\infty\|\sqrt{\rho}u_h\|_2^2\,d\tau.
\end{align*}
Applying Gronwall's inequality and using \eqref{estimate for low frequency velocity} give us that
\begin{align*}
\|u_h\|_{L_t^\infty(L^2)}+\|\nabla u_h\|_{L_t^2(L^2)}
\le C\|u_{0,h}\|_2\exp\{C\|u_{0,l}\|_{\dot{B}_{2,1}^{1/2}}\}.
\end{align*}
We interpolate to have
\begin{align*}
\|u_h\|_{L_t^4(\dot{B}_{2,1}^{1/2})}\le C\|u_{0,h}\|_2\exp\{C\|u_{0,l}\|_{\dot{B}_{2,1}^{1/2}}\}.
\end{align*}
This implies that there exists a positive number $t_\varepsilon$ such that $\|u_h(t_\varepsilon)\|_{\dot{B}_{2,1}^{1/2}}\le\varepsilon$, and so $\|u(t_\varepsilon)\|_{\dot{B}_{2,1}^{1/2}}\lesssim\varepsilon$. Now we  apply Lemma \ref{combine with zhang's result} to conclude that
\begin{align*}
\|u\|_{L^\infty((t_\varepsilon,\infty);\dot{B}_{2,1}^{1/2})}\lesssim \|u(t_\varepsilon)\|_{\dot{B}_{2,1}^{1/2}}\lesssim\varepsilon,
\end{align*}
which implies \eqref{large solution asymptotic to 0} since $\varepsilon$ is arbitrarily small. This completes the proof of Theorem \ref{long time asymptotics}.
\end{proof}

\bigskip


\appendix

\section{Equivalence between Eulerian and Lagrangian formulations}\label{appendix1}

\bigskip

We start with a classical result concerning the regularity of the solutions to the ODE
\begin{eqnarray}\label{ODE}
\left\{\begin{aligned}
&\frac{d}{dt}X(t,y)=u(t,X(t,y)),\\
&X(0,y)=y
\end{aligned}\right.
\end{eqnarray}
within the Cauchy-Lipschitz framework. Let us temporarily assume that $u$ is $C^1$ vector field, namely, $u\in L^1([0,T];C_b^{1}(\R^3;\R^3))$.

\begin{lemma}\label{cauchy lipschitz theorem}
For any $y\in\R^3$, \eqref{ODE} has a unique solution $X(\cdot,y)\in W^{1,1}([0,T])$. For any $t\in[0,T]$, $X(t,\cdot)$ is a $C^1$ diffeomorphism over $\R^3$ that satisfies
\begin{align*}
\|\nabla X(t)\|_{\infty}\vee\|\nabla X^{-1}(t)\|_{\infty}\le\exp\left(\|\nabla u\|_{L_t^1(L^\infty)}\right),
\end{align*}
where $X^{-1}(t,\cdot)$ is the inverse of $X(t,\cdot)$. The determinant $J_X(t,y)$ of $DX(t,y)$ satisfies
\begin{align*}
\exp\left(-\|\divg u\|_{L_t^1(L^\infty)}\right)\le J_X(t,y)\le\exp\left(\|\divg u\|_{L_t^1(L^\infty)}\right).
\end{align*}
In particular, if $\divg u=0$, then $J_X(t,y)$ is identical to $1$.
\end{lemma}

\noindent\fbox{From Eulerian to Lagrangian formulation.}
We assume that $(\rho,u,\nabla P)$ is a strong solution to \eqref{INS} (see Definition \ref{definition of strong solutions to Eulerian formulation}). In particular, $u$ is a $C^1$ vector field since $\dot{B}_{2,1}^{3/2}(\R^3)\hookrightarrow C_0(\R^3)$. Let $X_u(t,y)$ be the solution to \eqref{ODE}. Then we introduce new unknowns
\begin{align*}
(\lagr{\rho},\lagr{u},\lagr{P})(t,y)=(\rho,u,P)\big(t,X_u(t,y)\big).
\end{align*}
In view of the transport equation in \eqref{INS}, we have $\lagr{\rho}\equiv\rho_0$. Before we derive the equations satisfied by $(\lagr{u},\lagr{P})$, we have to study its regularity. To this end, we need the following result which guarantees that Besov regularity of a function is preserved under changes of variables.
\begin{lemma}\label{regularity under changes of variables}
Let $X$ be a $C^1$ diffeomorphism over $\R^3$. Let $p\in[1,\infty)$, $q\in[1,\infty]$ and $s\in(0,1)$. It holds that
\begin{align*}
\|f\circ X\|_{\dot{B}_{p,q}^{s}}\le C\|J_{X^{-1}}\|_{\infty}^{2/p}\|DX\|_{\infty}^{s+3/p}\|f\|_{\dot{B}_{p,q}^{s}},\ \ \forall f\in \mathscr{S}(\R^3).
\end{align*}
If in addition $DX-I\in\dot{B}_{6,1}^{1/2}(\R^3)$, we have
\begin{align*}
\|f\circ X\|_{\dot{B}_{2,1}^{3/2}}\le C(\|DX-I\|_{\dot{B}_{6,1}^{1/2}}+1)\|J_{X^{-1}}\|_{\infty}\|DX\|_{\infty}^{2}\|f\|_{\dot{B}_{2,1}^{3/2}},\ \ \forall f\in \dot{B}_{2,1}^{3/2}(\R^3).
\end{align*}
\end{lemma}
\begin{remark}
See {\cite[Lemma~2.1.1]{danchin memoir 2015}} for a more general result.
\end{remark}

Given Lemma \ref{regularity under changes of variables}, we are now able to improve the regularity for the trajectory $X_u(t,y)$.

\begin{lemma}\label{Besov regularity of the flow}
We have $DX_u-I\in W^{1,1}([0,T];\dot{B}_{2,1}^{3/2}(\R^3))$. There exists a constant $C$ depending on $\|Du\|_{L_T^1(\dot{B}_{2,1}^{3/2})}$ such that
\begin{align*}
\|DX_u-I\|_{L_T^\infty(\dot{B}_{2,1}^{3/2})}\le C\|Du\|_{L_T^1(\dot{B}_{2,1}^{3/2})}.
\end{align*}
\end{lemma}
\begin{proof}
Let us only show the estimate. Note that $DX_u$ satisfies the integral equation
\begin{align*}
DX_u(t,y)-I=\int_0^t Du(\tau,X_u(\tau,y))DX_u(\tau,y)\,d\tau.
\end{align*}
In order to obtain the global-in-time estimate without assuming smallness on $\|Du\|_{L_T^1(\dot{B}_{2,1}^{3/2})}$, we first need to obtain an estimate for $\|DX-I\|_{L_T^\infty(\dot{B}_{6,1}^{1/2})}$. Applying Lemma \ref{cauchy lipschitz theorem} and Lemma \ref{regularity under changes of variables}, and using the fact that $\dot{B}_{6,1}^{1/2}(\R^3)$ is an algebra, we arrive at
\begin{align*}
\|DX_u(t)-I\|_{\dot{B}_{6,1}^{1/2}}\le C\int_0^t\|Du(\tau)\|_{\dot{B}_{6,1}^{1/2}}\left(\|D X_u(\tau)-I\|_{\dot{B}_{6,1}^{1/2}}+1\right)\,d\tau,
\end{align*}
where $C$ depends on $\|\nabla u\|_{L_t^1(L^\infty)}$. So using Gronwall's inequality gives rise to
\begin{align*}
\|DX_u(t)-I\|_{\dot{B}_{6,1}^{1/2}}\le C\|Du\|_{L_t^1(\dot{B}_{6,1}^{1/2})}\exp\left(C\|Du\|_{L_t^1(\dot{B}_{6,1}^{1/2})}\right).
\end{align*}
Using a similar argument, then applying the above estimate, we get
\begin{align*}
\|DX_u(t)-I\|_{\dot{B}_{2,1}^{3/2}}\le C\int_0^t\|Du(\tau)\|_{\dot{B}_{2,1}^{3/2}}\left(\|D X_u(\tau)-I\|_{\dot{B}_{2,1}^{3/2}}+1\right)\,d\tau,
\end{align*}
where $C$ depends on $\|\nabla u\|_{L_t^1(\dot{B}_{6,1}^{1/2})}$. Using again Gronwall's inequality finishes the proof.
\end{proof}

Next, by the chain rule, Lemma \ref{regularity under changes of variables}, Lemma \ref{Besov regularity of the flow} and product laws in Besov spaces, we obtain
\begin{eqnarray}\label{regularity of Lagrangian variables in terms of Eulerian variables}
\left\{\begin{aligned}
&\|\lagr{u}\|_{L_T^\infty(\dot{B}_{2,1}^{1/2})}\le C\|u\|_{L_T^\infty(\dot{B}_{2,1}^{1/2})},\\
&\|\partial_t\lagr{u}\|_{L_T^1(\dot{B}_{2,1}^{1/2})}\le C\big(\|\partial_tu\|_{L_T^1(\dot{B}_{2,1}^{1/2})}+\|u\|_{L_T^2(\dot{B}_{2,1}^{3/2})}^2\big),\\
&\|\nabla\lagr{u}\|_{L_T^1(\dot{B}_{2,1}^{3/2})}\le C\|\nabla u\|_{L_T^1(\dot{B}_{2,1}^{3/2})},\\
&\|\nabla\lagr{P}\|_{L_T^1(\dot{B}_{2,1}^{1/2})}\le C\|\nabla P\|_{L_T^1(\dot{B}_{2,1}^{1/2})},
\end{aligned}\right.
\end{eqnarray}
where $C$ depends on $\|\nabla u\|_{L_T^1(\dot{B}_{2,1}^{3/2})}$. Using a similar argument and recalling the definition of adjugate matrices, it is not difficult to see that
\begin{align}\label{regularity of adjugate matrix}
\|\mathscr{A}_{\lagr{u}}-I\|_{L_T^\infty(\dot{B}_{2,1}^{3/2})}\le C\|\nabla u\|_{L_T^1(\dot{B}_{2,1}^{3/2})},
\end{align}
where, of course, $C$ depends on $\|\nabla u\|_{L_T^1(\dot{B}_{2,1}^{3/2})}$. Finally, the estimates \eqref{regularity of Lagrangian variables in terms of Eulerian variables} and \eqref{regularity of adjugate matrix} allow us to conclude that $(\lagr{u},\lagr{P})$ is a strong solution to \eqref{Lagrangian formulation} (see Definition \ref{definition of strong solutions to Lagrangian formulation}).

\noindent\fbox{From Lagrangian to Eulerian formulation.} 
Assume that $(\lagr{u},\nabla\lagr{P})$ is a strong solution to the system \eqref{Lagrangian formulation} supplemented with \eqref{definition of Au and Xu}. Unlike the previous justification in which the existence time for \eqref{Lagrangian formulation} is the same as that for \eqref{INS} unconditionally (i.e., without smallness condition on $\|\nabla u\|_{L_T^1(\dot{B}_{2,1}^{3/2})}$), here we shall assume
\begin{align}
\|\nabla\lagr{u}\|_{L_T^1(\dot{B}_{2,1}^{3/2})}\le c_0,
\end{align}
so that $X_{\lagr{u}}(t,\cdot)$ (defined in \eqref{definition of Au and Xu}) becomes a $C^1$ diffeomorphism over $\R^3$ for every $t\in[0,T]$. In fact, we have the following:

\begin{lemma}[see also \cite{danchin cpam 2012}]\label{X largv is C1 diffeomorphism}
Suppose that $\lagr{v}\in L^1([0,T];C_b^{1}(\R^3;\R^3))$ satisfies
\begin{align}\label{smallness on lip norm of v}
\|\nabla\lagr{v}\|_{L_T^1(L^\infty)}\le\frac12.
\end{align}
Then $X_{\lagr{v}}(t,\cdot)$ is a $C^1$-diffeomorphism over $\R^3$ for every $t\in[0,T]$. Denote by $X_{\lagr{v}}^{-1}(t,\cdot)$ the inverse of $X_{\lagr{v}}(t,\cdot)$. It holds that
\begin{align}
\|DX_{\lagr{v}}-I\|_{L_T^\infty(L^\infty)}\le& \|D\lagr{v}\|_{L_T^1(L^\infty)},\label{L infty norm of DXv-I}\\
\|DX_{\lagr{v}}^{-1}-I\|_{L_T^\infty(L^\infty)}\le& 2\|D\lagr{v}\|_{L_T^1(L^\infty)}\label{L infty norm of DYv-I}.
\end{align}
\end{lemma}
\begin{proof}
The first inequality \eqref{L infty norm of DXv-I} is easily seen. Let $Y_{\lagr{v}}(t,x)$ be the solution of the integral equation $Y_{\lagr{v}}(t,x)=x-\int_0^t\lagr{v}(\tau,Y_{\lagr{v}}(t,x))\,d\tau$. Note that this equation is solvable under the assumption \eqref{smallness on lip norm of v}. Then it is not difficult to see that $Y_{\lagr{v}}(t,\cdot)$ and $X_{\lagr{v}}(t,\cdot)$ are inverses to each other. One can readily get $\|DY_{\lagr{v}}\|_{L_T^\infty(L^\infty)}\le2$, which further implies \eqref{L infty norm of DYv-I}. 
\end{proof}

Next, we assume additionally that $\lagr{v}\in C([0,T];\dot{B}_{2,1}^{1/2}(\R^3))\cap L^1((0,T);\dot{B}_{2,1}^{5/2}(\R^3))$ satisfies
\begin{align}\label{smallness on besov norm of v}
\|\nabla\lagr{v}\|_{L_T^1(\dot{B}_{2,1}^{3/2})}\le c_0,
\end{align}
so that \eqref{smallness on lip norm of v} is fulfilled. The number $c_0$ may be chosen even smaller later, but it is always an absolute constant. Let $\mathscr{A}_{\lagr{v}}$ be defined by \eqref{definition of Au and Xu}. Then we have the following:
\begin{lemma}[see also \cite{danchin cpam 2012}]\label{inverse flow estimates}
There exists a constant $C>0$ such that
\begin{align}
\|DX_{\lagr{v}}-I\|_{L_T^\infty(\dot{B}_{2,1}^{3/2})}+\|\mathscr{A}_{\lagr{v}}-I\|_{L_T^\infty(\dot{B}_{2,1}^{3/2})}\le& C\|D\lagr{v}\|_{L_T^1(\dot{B}_{2,1}^{3/2})},\label{flow estimate for A-I}\\
\|DX_{\lagr{v}}^{-1}-I\|_{L_T^\infty(\dot{B}_{2,1}^{3/2})}\le& C\|D\lagr{v}\|_{L_T^1(\dot{B}_{2,1}^{3/2})},\nonumber\\
\|\partial_tX_{\lagr{v}}^{-1}\|_{L_T^2(\dot{B}_{2,1}^{3/2})}\le& C\|\lagr{v}\|_{L_T^2(\dot{B}_{2,1}^{3/2})}\nonumber.
\end{align}
\end{lemma}
\begin{proof}
We only prove the last two estimates because they were not explicitly given in \cite{danchin cpam 2012}. Note that $X_{\lagr{v}}^{-1}(t,x)$ satisfies
\begin{align*}
DX_{\lagr{v}}^{-1}(t,x)-I=-\int_0^tD\lagr{v}(\tau,X_{\lagr{v}}^{-1}(t,x))\,d\tau DX_{\lagr{v}}^{-1}(t,x).
\end{align*}
Applying Lemma \ref{regularity under changes of variables}, Lemma \ref{X largv is C1 diffeomorphism} and the fact that $\dot{B}_{6,1}^{1/2}(\R^3)$ is an algebra, we have
\begin{align*}
\|DX_{\lagr{v}}^{-1}-I\|_{L_T^\infty(\dot{B}_{6,1}^{1/2})}\le C\|D\lagr{v}\|_{L_T^1(\dot{B}_{6,1}^{1/2})}(\|DX_{\lagr{v}}^{-1}-I\|_{L_T^\infty(\dot{B}_{6,1}^{1/2})}+1).
\end{align*}
From this, if $c_0$ is small enough, we get $\|DX_{\lagr{v}}^{-1}-I\|_{L_T^\infty(\dot{B}_{6,1}^{1/2})}\lesssim c_0$. This further implies the second estimate. Noticing that
\begin{align*}
\partial_tX_{\lagr{v}}^{-1}(t,x)=-DX_{\lagr{v}}^{-1}(t,x)\lagr{v}(t,X_{\lagr{v}}^{-1}(t,x)),
\end{align*}
then the third estimate follows from the second one. So the proof is completed.
\end{proof}

With Lemma \ref{inverse flow estimates} in hand, one can go back to the Euler coordinates by introducing
\begin{align*}
\rho(t,x)=\rho_0(X_{\lagr{u}}^{-1}(t,x))\ \ \mathrm{and}\ \ (u,P)(t,x)=(\lagr{u},\lagr{P})(t,X_{\lagr{u}}^{-1}(t,x)).
\end{align*}
Then by the chain rule, Lemma \ref{regularity under changes of variables}, Lemma \ref{inverse flow estimates} and product laws in Besov spaces, we obtain
\begin{eqnarray*}
\left\{\begin{aligned}
&\|u\|_{L_T^\infty(\dot{B}_{2,1}^{1/2})}\le C\|\lagr{u}\|_{L_T^\infty(\dot{B}_{2,1}^{1/2})},\\
&\|\partial_tu\|_{L_T^1(\dot{B}_{2,1}^{1/2})}\le C\big(\|\partial_t\lagr{u}\|_{L_T^1(\dot{B}_{2,1}^{1/2})}+\|\lagr{u}\|_{L_T^2(\dot{B}_{2,1}^{3/2})}^2\big),\\
&\|\nabla u\|_{L_T^1(\dot{B}_{2,1}^{3/2})}\le C\|\nabla\lagr{u}\|_{L_T^1(\dot{B}_{2,1}^{3/2})},\\
&\|\nabla P\|_{L_T^1(\dot{B}_{2,1}^{1/2})}\le C\|\nabla\lagr{P}\|_{L_T^1(\dot{B}_{2,1}^{1/2})}.
\end{aligned}\right.
\end{eqnarray*}
For the estimates of the density, we can use a classical result for transport equations (see, e.g., {\cite[Proposition~3.1]{danchin edinburgh 2003}}) to conclude that $\rho-1\in C([0,T];\dot{B}_{2,1}^{3/2}(\R^3))$, and that
\begin{align*}
\|\rho-1\|_{L_T^\infty(\dot{B}_{2,1}^{3/2})}\le\|\rho_0-1\|_{\dot{B}_{2,1}^{3/2}}\exp(C\|\nabla u\|_{L_T^1(\dot{B}_{2,1}^{3/2})})\le C\|\rho_0-1\|_{\dot{B}_{2,1}^{3/2}}.
\end{align*}
Consequently, we also have $\partial_t\rho\in L^2((0,T);\dot{B}_{2,1}^{1/2}(\R^3))$. Finally, one can readily check that $(\rho,u,P)$ is a strong solution to \eqref{INS}.

We conclude this section with some estimates that have been used to prove the existence and stability of solutions to the Lagrangian formulation \eqref{Lagrangian formulation}.

\begin{lemma}[see \cite{danchin cpam 2012}]\label{flow estimates}
Let $\lagr{v}, \lagr{v}_1$ and $\lagr{v}_2$ be vector fields in $C([0,T];\dot{B}_{2,1}^{1/2}(\R^3))\cap L_T^1(\dot{B}_{2,1}^{5/2}(\R^3))$ and satisfy \eqref{smallness on besov norm of v}. Let $\delta\lagr{v}=\lagr{v}_1-\lagr{v}_2$. Then we have
\begin{align}
\|\partial_t\mathscr{A}_{\lagr{v}}(t)\|_{\dot{B}_{2,1}^{s}}\lesssim&\|\nabla\lagr{v}(t)\|_{\dot{B}_{2,1}^{s}},\ \ s=\frac12,\frac32,\label{flow estimate for partial t A}\\
\|\mathscr{A}_{\lagr{v}_1}-\mathscr{A}_{\lagr{v}_2}\|_{L_t^\infty(\dot{B}_{2,1}^{3/2})}\lesssim&\|\nabla\delta\lagr{v}\|_{L_t^1(\dot{B}_{2,1}^{3/2})},\label{flow estimate for difference of A}\\
\|\partial_t(\mathscr{A}_{\lagr{v}_1}-\mathscr{A}_{\lagr{v}_2})\|_{L_t^2(\dot{B}_{2,1}^{1/2})}\lesssim&\|\delta\lagr{v}\|_{L_t^2(\dot{B}_{2,1}^{3/2})}\label{flow estimate for partial t difference of A L2}.
\end{align}
\end{lemma}


\bigskip
\bigskip
\bigskip

\begin{thebibliography}{99}

\bibitem{abidi rmi 2007}
\newblock H. Abidi,
\newblock \'{E}quation de Navier-Stokes avec densit\'{e} et viscosit\'{e} variables dans l'espace critique,
\newblock \emph{Rev. Mat. Iberoam.} \textbf{23} (2007), no. 2, 537-586.

\bibitem{abidi arma 2012}
\newblock H. Abidi, G. Gui, P. Zhang,
\newblock On the wellposedness of three-dimensional inhomogeneous Navier-Stokes equations in the critical spaces,
\newblock \emph{Arch. Ration. Mech. Anal.} \textbf{204} (2012), no. 1, 189-230.

\bibitem{abidi jmpa 2013}
\newblock H. Abidi, G. Gui, P. Zhang,
\newblock Well-posedness of 3-D inhomogeneous Navier-Stokes equations with highly oscillatory initial velocity field,
\newblock \emph{J. Math. Pures Appl. (9)} \textbf{100} (2013), no. 2, 166-203.

\bibitem{abidi aif 2007}
\newblock H. Abidi, M. Paicu,
\newblock Existence globale pour un fluide inhomog\`{e}ne,
\newblock \emph{Ann. Inst. Fourier (Grenoble)} \textbf{57} (2007), no. 3, 883-917.

\bibitem{arendt book 2011}
\newblock  W. Arendt, C. J. K. Batty, M. Hieber, F. Neubrander,
\newblock \emph{Vector-valued Laplace transforms and Cauchy problems},
\newblock Second edition. Monographs in Mathematics, 96. Birkh\"{a}user/Springer Basel AG, Basel, 2011.

\bibitem{auscher jlms 1996}
\newblock P. Auscher,
\newblock Regularity theorems and heat kernel for elliptic operators,
\newblock \emph{J. London Math. Soc. (2)} \textbf{54} (1996), no. 2, 284-296.

\bibitem{auscher book 1998}
\newblock P. Auscher, P. Tchamitchian,
\newblock \emph{Square root problem for divergence operators and related topics},
\newblock Ast\'{e}risque No. 249 (1998).

\bibitem{BCD}
\newblock  H. Bahouri, J.-Y. Chemin, R. Danchin,
\newblock \emph{Fourier analysis and nonlinear partial differential equations},
\newblock Grundlehren der Mathematischen Wissenschaften [Fundamental Principles of Mathematical Sciences], \textbf{343}, Springer, Heidelberg, 2011.

\bibitem{bui adv 2012}
\newblock H.-Q. Bui, X. T. Duong, L. Yan,
\newblock Calder\'{o}n reproducing formulas and new Besov spaces associated with operators,
\newblock \emph{Adv. Math.}
\textbf{229} (2012), no. 4, 2449-2502.

\bibitem{burtea apde 2017}
\newblock C. Burtea,
\newblock Optimal well-posedness for the inhomogeneous incompressible Navier-Stokes system with general viscosity,
\newblock \emph{Anal. PDE}
\textbf{10} (2017), no. 2, 439-479.

\bibitem{chen jde 2016}
\newblock D. Chen, Z. Zhang, W. Zhao,
\newblock Fujita-Kato theorem for the 3-D inhomogeneous Navier-Stokes equations,
\newblock \emph{J. Differential Equations}, \textbf{261} (2016), no. 1, 738-761.

\bibitem{danchin cpde 2001}
\newblock R. Danchin,
\newblock Local theory in critical spaces for compressible viscous and heat-conductive gases,
\newblock \emph{Comm. Partial Differential Equations}, \textbf{26} (2001), no. 7-8, 1183-1233.

\bibitem{danchin edinburgh 2003}
\newblock R. Danchin,
\newblock Density-dependent incompressible viscous fluids in critical spaces,
\newblock \emph{Proc. Roy. Soc. Edinburgh Sect. A}, \textbf{133} (2003), no. 6, 1311-1334.


\bibitem{danchin arxiv 2020}
\newblock R. Danchin,  M. Hieber, P. B. Mucha, P. Tolksdorf,
\newblock Free boundary problems via Da Prato-Grisvard theory,
\newblock  arXiv:2011.07918.

\bibitem{danchin jfa 2009}
\newblock R. Danchin, P. B. Mucha,
\newblock A critical functional framework for the inhomogeneous Navier-Stokes equations in the half-space,
\newblock \emph{J. Funct. Anal.} \textbf{256} (2009), no. 3, 881-927.

\bibitem{danchin cpam 2012}
\newblock R. Danchin, P. B. Mucha,
\newblock A Lagrangian approach for the incompressible Navier-Stokes equations with variable density,
\newblock \emph{Comm. Pure Appl. Math.} \textbf{65} (2012), no. 10, 1458-1480.

\bibitem{danchin arma 2013}
\newblock R. Danchin, P. B. Mucha,
\newblock Incompressible flows with piecewise constant density,
\newblock \emph{Arch. Ration. Mech. Anal.} \textbf{207} (2013), no. 3, 991-1023.

\bibitem{danchin memoir 2015}
\newblock R. Danchin, P. B. Mucha,
\newblock Critical functional framework and maximal regularity in action on systems of incompressible flows,
\newblock \emph{M\'{e}m. Soc. Math. Fr. (N.S.)} (2015), no. 143, vi+151 pp.

\bibitem{danchin cpam 2019}
\newblock R. Danchin, P. B. Mucha,
\newblock The incompressible Navier-Stokes equations in vacuum,
\newblock \emph{Comm. Pure Appl. Math.} \textbf{72} (2019), no. 7, 1351-1385.

\bibitem{desjardins DIE 1997}
\newblock B. Desjardins,
\newblock Global existence results for the incompressible density-dependent Navier-Stokes equations in the whole space,
\newblock \emph{Differential Integral Equations},
\textbf{10} (1997), no. 3, 587-598.

\bibitem{duong DIE 1999}
\newblock X. T. Duong, E. M. Ouhabaz,
\newblock Complex multiplicative perturbations of elliptic operators: heat kernel bounds and holomorphic functional calculus,
\newblock \emph{Differential Integral Equations},
\textbf{12} (1999), no. 3, 395-418.

\bibitem{engel book 2000}
\newblock K.-J. Engel, R. Nagel,
\newblock \emph{One-parameter semigroups for linear evolution equations},
\newblock Graduate Texts in Mathematics, 194. Springer-Verlag, New York, 2000.

\bibitem{gallagher 2002}
\newblock I. Gallagher, D. Iftimie, F.  Planchon,
\newblock Non-explosion en temps grand et stabilit\'{e} de solutions globales des \'{e}quations de Navier-Stokes,
\newblock \emph{C. R. Math. Acad. Sci. Paris}, \textbf{334} (2002), no. 4, 289-292.

\bibitem{gallagher 2003}
\newblock I. Gallagher, D. Iftimie, F.  Planchon,
\newblock Asymptotics and stability for global solutions to the Navier-Stokes equations,
\newblock \emph{Ann. Inst. Fourier (Grenoble)}, \textbf{53} (2003), no. 5, 1387-1424.

\bibitem{liu fm 2015}
\newblock A. Grigor'yan, L. Liu,
\newblock Heat kernel and Lipschitz-Besov spaces,
\newblock \emph{Forum Math.} \textbf{27} (2015), no. 6, 3567-3613.

\bibitem{huang arma 2013}
\newblock J. Huang, M. Paicu, P. Zhang,
\newblock Global well-posedness of incompressible inhomogeneous fluid systems with bounded density or non-Lipschitz velocity,
\newblock \emph{Arch. Ration. Mech. Anal.} \textbf{209} (2013), no. 2,  631-682.

\bibitem{jiang jmpa 2020}
\newblock R. Jiang, F. Lin,
\newblock Riesz transform under perturbations via heat kernel regularity,
\newblock \emph{J. Math. Pures Appl.}
\textbf{133} (2020), no. 9, 39-65.

\bibitem{kazihov 1974}
\newblock A. V. Ka\v{z}ihov,
\newblock Solvability of the initial-boundary value problem for the equations of the motion of an inhomogeneous viscous incompressible fluid,
\newblock \emph{Dokl. Akad. Nauk SSSR},
\textbf{216} (1974), 1008-1010.

\bibitem{ladyzhenskaya}
\newblock O. A. Lady\v{z}enskaja, V. A. Solonnikov,
\newblock The unique solvability of an initial-boundary value problem for viscous incompressible inhomogeneous fluids,
\newblock \emph{Zap. Nau\v{c}n. Sem. Leningrad. Otdel. Mat. Inst. Steklov. (LOMI)}
\textbf{52} (1975), 52-109, 218-219.

\bibitem{Lions book}
\newblock   P.-L. Lions,
\newblock
\emph{Mathematical topics in fluid mechanics. Vol. 1. Incompressible models},
\newblock Oxford Lecture Series in Mathematics and its Applications, 3. Oxford Science Publications. The Clarendon Press, Oxford University Press, New York, 1996.

\bibitem{mcintosh 2000}
\newblock A. McIntosh, A. Nahmod,
\newblock Heat kernel estimates and functional calculi of $-b\Delta$,
\newblock \emph{Math. Scand.}
\textbf{87} (2000), no. 2, 287-319.

\bibitem{ouhabaz pams 1995}
\newblock E.-M. Ouhabaz,
\newblock Gaussian estimates and holomorphy of semigroups,
\newblock \emph{Proc. Amer. Math. Soc.}
\textbf{123} (1995), no. 5, 1465-1474.

\bibitem{paicu cpde 2013}
\newblock M. Paicu, P. Zhang, Z. Zhang,
\newblock Global unique solvability of inhomogeneous Navier-Stokes equations with bounded density,
\newblock \emph{Comm. Partial Differential Equations}, \textbf{38} (2013), no. 7, 1208-1234.

\bibitem{pazy book 1983}
\newblock A. Pazy,
\newblock \emph{Semigroups of linear operators and applications to partial differential equations},
\newblock Applied Mathematical Sciences, 44. Springer-Verlag, New York, 1983.


\bibitem{simon siam 1990}
\newblock   J. Simon,
\newblock Nonhomogeneous viscous incompressible fluids: existence of velocity, density, and pressure,
\newblock \emph{SIAM J. Math. Anal.} \textbf{21} (1990), no. 5, 1093-1117.

\bibitem{solonnikov}
\newblock V. A. Solonnikov,
\newblock On a nonstationary motion of an isolated mass of a viscous incompressible fluid (in Russian),
\newblock \emph{Izv. Akad. Nauk SSSR Ser. Mat.} \textbf{51} (1987), no. 5, 1065-1087; \emph{translation in
Math. USSR-Izv.} \textbf{31} (1988), no. 2, 381-405.

\bibitem{Triebel book 1983}
\newblock H. Triebel,
\newblock \emph{Theory of function spaces},
\newblock Monographs in Mathematics, 78. Birkh\"{a}user Verlag, Basel, 1983. 284 pp.

\bibitem{zhai jde 2017}
\newblock   X. Zhai, Z. Yin,
\newblock Global well-posedness for the 3D incompressible inhomogeneous Navier-Stokes equations and MHD equations,
\newblock \emph{J. Differential Equations} \textbf{262} (2017), no. 3, 1359-1412.

\bibitem{zhang ping adv 2020}
\newblock   P. Zhang,
\newblock Global Fujita-Kato solution of 3-D inhomogeneous incompressible Navier-Stokes system,
\newblock \emph{Adv. Math.} \textbf{363} (2020), 107007, 43 pp.

\end{thebibliography}
\end{document}